\documentclass[reqno]{amsart}

\RequirePackage{fix-cm}
\usepackage{graphicx}
\usepackage{mathptmx}      % use Times fonts if available on your TeX system
\usepackage{latexsym}
\RequirePackage[OT1]{fontenc}
\usepackage{latexsym,amsmath}
\usepackage{amsmath,%amsthm,
amscd}
\usepackage{amsfonts}
\usepackage[psamsfonts]{amssymb}
\usepackage{enumerate}

\usepackage{url}
\usepackage{tocvsec2}

\usepackage[final]
{showkeys}
\usepackage{color}
\usepackage{bm}

\usepackage{float}
\usepackage{array}

\usepackage{stmaryrd}

\usepackage{hyperref}%

\newtheorem{theorem}{Theorem}[section]
\newtheorem*{theorem*}{Theorem (Lehmann--Scheff\'e)}
\newtheorem{lemma}[theorem]{Lemma}
\newtheorem{corollary}[theorem]{Corollary}
\newtheorem{proposition}[theorem]{Proposition}

\theoremstyle{definition}

\newtheorem{remark}[theorem]{Remark}

\theoremstyle{definition} 
\newtheorem*{remark*}{Remark}

\numberwithin{equation}{section}

\newcommand{\R}{\mathbb{R}}
\newcommand{\N}{\mathbb{N}}

\newcommand{\Z}{\mathbb{Z}}
\newcommand{\CC}{\mathbb{C}}

\newcommand{\JJ}{\mathcal{J}}

\newcommand{\LL}{\mathcal{L}}
\newcommand{\RR}%{T(X)}%{R_T}%
{\mathcal{R}}

\newcommand{\tA}{\tilde A}

\newcommand{\lb}{\llbracket}
\newcommand{\rb}{\rrbracket}

\newcommand{\al}{\alpha}
\newcommand{\be}{\beta}
\newcommand{\ga}{\gamma}
\newcommand{\Ga}{\Gamma}
\newcommand{\de}{\delta}
\newcommand{\De}{\Delta}
\newcommand{\vp}{\varepsilon}
\newcommand{\ka}{\kappa}
\newcommand{\la}{\lambda}
\newcommand{\La}{\Lambda}

\newcommand{\0}{\mathbf{0}}
\newcommand{\1}{\mathbf{1}}

\newcommand{\aal}{{\bm{\al}}}
\newcommand{\bbe}{{\bm{\be}}}
\newcommand{\lla}{{\bm{\la}}}

\renewcommand{\aa}{\mathbf{a}}
\newcommand{\cc}{\mathbf{c}}
\newcommand{\iii}{\mathbf{i}}
\newcommand{\jj}{\mathbf{j}}
\newcommand{\kk}{\mathbf{k}}
\newcommand{\elll}{\mathbf{q}}
\newcommand{\nn}{\mathbf{n}}
\newcommand{\uu}{\mathbf{u}}
\newcommand{\vv}{\mathbf{v}}
\newcommand{\ww}{\mathbf{w}}
\newcommand{\xx}{\mathbf{x}}
\newcommand{\yy}{\mathbf{y}}

\newcommand{\dd}{\operatorname{d}\!}

\newcommand{\ii}{\operatorname{I}}

\newcommand{\Alt}{\mathsf{Alt}}

\begin{document}

\title[Approximating multiple sums by multiple integrals only]{%An alternative to the Euler--Maclaurin formula: \\ 
Approximating sums by integrals only: \\ 
multiple sums and sums over lattice polytopes}

%\title[Completeness and sufficiency of UMVUEs]{On the completeness and sufficiency of uniformly minimum variance unbiased estimators}

%%
%% Now edit the following to give your name and address:
%% 

\author{Iosif Pinelis}
\address{Department of Mathematical Sciences, Michigan Technological University}
\email{ipinelis@mtu.edu}
%\urladdr{www.math.sc.edu/$\sim$howard} % Delete if not wanted.

%, 
%68Q17
%}

% 	26A06   	One-variable calculus
% 	26A36   	Antidifferentiation

%%%26B20   	Integral formulas (Stokes, Gauss, Green, etc.)
%%%
%%%26D15   	Inequalities for sums, series and integrals
%%%
%%%40A05   	Convergence and divergence of series and sequences
%%% 	40A25   	Approximation to limiting values (summation of series, etc.) {For the Euler-Maclaurin summation formula, see 65B15}
% 	
% 	41A10   	Approximation by polynomials {For approximation by trigonometric polynomials, see 42A10}
% 	 	41A17   	Inequalities in approximation (Bernstein, Jackson, Nikol'skii-type inequalities)
% 	 	41A25   	Rate of convergence, degree of approximation
% 	 	 	41A35   	Approximation by operators (in particular, by integral operators)
% 	41A55   	Approximate quadratures
% 41A58   	Series expansions (e.g. Taylor, Lidstone series, but not Fourier series)
%		41A60   	Asymptotic approximations, asymptotic expansions (steepest descent, etc.) [See also 30E15]
%		 	41A80   	Remainders in approximation formulas

%%% 	52B20   	Lattice polytopes (including relations with commutative algebra and algebraic geometry) [See also 06A11, 13F20, 13Hxx]
 	
%		 	
%		 	65D10   	Smoothing, curve fitting
%		 	65D30   	Numerical integration
%		65D32   	Quadrature and cubature formulas	
%		
%		68Q17   	Computational difficulty of problems (lower bounds, completeness, difficulty of approximation, etc.) [See also 68Q15]

%41A58, 41A60, 41A80, 65D30, 68Q17
 	

\begin{abstract} 
The Euler--Maclaurin (EM) summation formula is used in many theoretical studies and numerical calculations. It approximates the sum 
$\sum_{k=0}^{n-1} f(k)$ of values of a function $f$ by a linear combination of a corresponding integral of $f$ and values of its higher-order derivatives $f^{(j)}$. 
An alternative (Alt) summation formula was recently presented by the author, which approximates the sum by a linear combination of integrals only,  
without using high-order derivatives of $f$. 
It was shown that the Alt formula will in most cases outperform, or greatly outperform, the EM formula in terms of the execution time and memory use. 
In the present paper, a multiple-sum/multi-index-sum extension of the Alt formula is given, with  
applications to summing possibly divergent multi-index series and  
to sums over the integral points of integral lattice polytopes. 
\end{abstract}

\maketitle

\setcounter{tocdepth}{1}
\tableofcontents

\section{Introduction}\label{intro}

The Euler--Maclaurin (EM) summation formula
can be written as follows: 
\begin{equation*}\label{eq:EMintro}
	\sum_{k=0}^{n-1} f(k)\approx 
	\int^n_0\dd x\, f(x)  
    +
    \sum_{j=1}^{2m-1}\frac{B_j}{j!}[f^{(j-1)}(n) - f^{(j-1)}(0)], \tag{EM}      
\end{equation*} 
where $f\colon\R\to\R$ is a smooth enough function, 
$B_{j}$ is the $j$th Bernoulli number, and $n$ and $m$ are natural numbers. The EM approximation is exact when $f$ is a polynomial of degree $<2m-1$. 
%The formula was discovered independently by Leonhard Euler and Colin Maclaurin around 1735. 
%%%The integral $\int^n_0\dd x\, f(x)$ in \eqref{eq:EMintro} may be considered the main term of the approximation, whereas the summands $\frac{B_j}{j!}[f^{(j-1)}(n) - f^{(j-1)}(0)]$ may be referred to as the correction terms.  
% 
The EM formula has been used %very broadly 
in a large number of theoretical studies and numerical calculations. 
Clearly, to use the EM formula in a theoretical or computational study, one will usually need to have an  antiderivative $F$ of $f$ and the derivatives $f^{(j-1)}$ for $j=1,\dots,2m-1$ in tractable or, respectively, computable form. 

In \cite{euler-maclaurin-alt}, an alternative summation formula (Alt) was offered, which approximates the sum $\sum_{k=0}^{n-1} f(k)$ by a linear combination of values of an antiderivative $F$ of $f$ only, without using values of any derivatives of $f$: 
\begin{equation*}\label{eq:intro}
	\sum_{k=0}^{n-1} f(k)\approx 
	\sum_{j=1-m}^{m-1}\tau_{m,1+|j|}\,\int_{j/2-1/2}^{n-1/2-j/2}\dd x\, f(x), \tag{Alt}  
\end{equation*}
where $f$ is again a smooth enough function, 
the coefficients $\tau_{m,r}$ are certain rational numbers not depending on $f$ and such that $\sum_{j=1-m}^{m-1}\tau_{m,1+|j|}=1$, and $n$ and $m$ are natural numbers. %IP such that $n\ge m-1$. 
Similarly to the case of the EM formula, the Alt approximation is exact when $f$ is a polynomial of degree $<2m$. 
It was shown in \cite{euler-maclaurin-alt} that the Alt formula should be usually expected to outperform the EM one.  

Extensions of the EM formula to the multiple sums, including sums over the integral points of integral lattice polytopes, have been of significant interest; see e.g.\ \cite{KSW_ProcNAS,KSW_duke07}. In the present paper, a multiple-sum/multi-index-sum extension of the Alt formula will be given. 
The main result of this paper,  
Theorem~\ref{th:}, is then extended to sums over the integral points of integral lattice polytopes as well. 
%, using inclusion-exclusion types of formulas expressing the indicator function of a polytope as a linear combination (with coefficients of the form $(-1)^q$ with integer $q$) of the indicator functions of (finitely generated) cones; cf.\ e.g.\ 
%\cite{lawrence,KSW_duke07,agap-godin16}. 
%Actually, formulas \eqref{eq:FTC} and \eqref{eq:incl-excl}, used in the proofs of Theorems~\ref{prop:series} and \ref{th:c}, are simplest examples of such inclusion-exclusion formulas. 

\medskip

The rest of this paper is organized as follows.  

In Section~\ref{result}, the multi-index Alt formula is %IP06-02-17	rigorously 
stated, with discussion. 

In Section~\ref{series}, an application of the Alt %and EM 
formula %s 
to summing possibly divergent multi-index series is given. 
A shift trick then allows one to make the remainder in the Alt formula arbitrarily small. 

In Section~\ref{polytopes}, the mentioned extension to sums over the integral points of integral lattice polytopes is presented. 

The necessary proofs are deferred to Section~\ref{proof}. 

\medskip

At the end of this introduction, let us fix notation to be used in the rest of the paper:  
 %IP multiple edits in this paragraph 
Suppose that $p$ and $m$ are natural numbers and $f\colon\R^p\to\R$ is a $2m$-times continuously differentiable function, with partial derivatives $f^{(\aal)}$, where $\aal=(\al_1,\dots,\al_p)\in\Z_+^p$ and $\Z_+:=\Z\cap[0,\infty)$. 
%Let $\|\aal\|:=\|\aal\|_1=%\break 
%\al_1+\dots+\al_p%\le2m
%$ and $\aal!:=\al_1!\cdots\al_p!$. 
Generally, boldface letters will denote vectors in $\R^p$, in $\Z^p$, or in $\Z_+^p$, with the coordinates denoted by the corresponding non-boldface letters with the indices: $\xx=(x_1,\dots,x_p)\in\R^p$, $\uu=(u_1,\dots,u_p)\in\R^p$, $\vv=(v_1,\dots,v_p)\in\R^p$, $\nn=(n_1,\dots,n_p)\in\Z_+^p$, $\kk=(k_1,\dots,k_p)\in\Z_+^p$, $\jj=(j_1,\dots,j_p)\in\Z_+^p$, $\iii=(i_1,\dots,i_p)\in\Z_+^p$, $\aal=(\al_1,\dots,\al_p)\in\Z_+^p$,  
and $\bbe=(\be_1,\dots,\be_p)\in\Z^p$. 
Let $\ii\{A\}$ denote the indicator of an assertion $A$. 
Let $\|\aal\|:=\|\aal\|_1=%\break 
\al_1+\dots+\al_p%\le2m
$; $\aal!:=\al_1!\cdots\al_p!$; $\xx^\aal:=x_1^{\al_1}\cdots x_p^{\al_p}$; $|\bbe|:=(|\be_1|,\dots,|\be_p|)$; $\1:=(1,\dots,1)\in\Z_+^p$; $\0:=0\1$; %$\|\jj\|:=\|\jj\|_\infty=\max(j_1,\dots,j_p)$, 
$\jj\vv:=(j_1v_1,\dots,j_pv_p)$; $\jj\ge\iii\mathrel{\overset{\text{def}}\iff} \iii\le\jj\mathrel{\overset{\text{def}}\iff} i_r\le j_r$ for $r=1,\dots,p$; %$\sum\limits_{\iii=\jj}^\kk:=\sum\limits_{\iii\in\Z_+^p\colon\jj\le\iii\le\kk}$,
%IP 
$[\uu,\vv]:=\prod_{r=1}^p[u_r,v_r]$; 
\break 
$\wedge\xx:=x_1\wedge\dots\wedge x_p$; 
$\vee\xx:= 
x_1\vee\dots\vee x_p$; 
$\uu\wedge\vv:=(u_1\wedge v_1,\dots,%\break 
u_p\wedge v_p)$; 
$\uu\vee\vv:=(u_1\vee v_1,\dots,u_p\vee v_p)$; 
\begin{equation*}
\sum\limits_{\iii=\jj}^\kk:=\sum\limits_{\iii\in\Z_+^p\colon\jj\le\iii\le\kk}; 
%\quad\text{and}
\qquad 
\int_\uu^\vv \dd\xx\; h(\xx):=%\int_\uu^\vv \dd \xx\,f(\xx):=\int_{(\uu,\vv]} \dd \xx\,f(\xx),
(-1)^{\sum_{r=1}^p\ii\{u_r>v_r\} }\int_{[\uu\wedge\vv,\uu\vee\vv]}\dd\xx\; h(\xx); \qquad 
\int_\uu^\vv:=\int_\uu^\vv \dd \xx\,f(\xx). 
\end{equation*}
%where $\ii\{A\}$ denotes the indicator of an assertion $A$. 
%As usual, let $\de_\xx$ (respectively, $\de_x$) denote the Dirac measure at the point $\xx\in\R^p$ (respectively, at the point $x\in\R$). 
%Let $\ii\{A\}$ denote the indicator of an assertion $A$. 
Let $\R_+^p:=[0,\infty)^p$.  
%!! rather easy to tensorize 

\section{%Main result and discussion
A multi-index alternative (Alt) to the EM formula}\label{result}

The following extension 
of \cite[Theorem 3.1]{euler-maclaurin-alt} to multiple sums
is the main result of this paper: 

\begin{theorem}\label{th:}
%IP Suppose that $\nn\ge(m-1)\1$. % and any $f\in C^{2m-}$. 
%%
%Then 
One has 
\begin{equation}\label{eq:}
	\sum_{\kk=\0}^{\nn-\1}f(\kk)\Big[=\sum_{k_1=0}^{n_1-1}\dots\sum_{k_p=0}^{n_p-1}f(k_1,\dots,k_p)\Big]
	=A_m-R_m,
\end{equation}
where 
\begin{alignat}{2}
	A_m\ \ :=\ \ \ \ 
	&\sum_{\jj=\1}^{m\1}\ga_{m,\jj}\sum_{\iii=\0}^{\jj-\1}\int_{\iii-\jj/2}^{\nn-\1+\jj/2-\iii}%{\kern-5pt}\dd x\,f(x) 
	&&\ \ =\ \ \ \ \sum_{\jj=\1}^{m\1}\ga_{m,\jj}\sum_{\iii=\0}^{\jj-\1}\int_{-\1+\jj/2-\iii}^{\nn-\1+\jj/2-\iii}
	\label{eq:A_m} \\ 
	\ \ =\ \ \ \ &\sum_{\bbe=(1-m)\1}^{(m-1)\1}\tau_{m,\1+|\bbe|}\,\int_{\bbe/2-\1/2}^{\nn-\1/2-\bbe/2}
	&&\ \ =\ \ \ \ \sum_{\bbe=(1-m)\1}^{(m-1)\1}\tau_{m,\1+|\bbe|}\,\int_{-\1/2-\bbe/2}^{\nn-\1/2-\bbe/2}
	\label{eq:A_m,alt1} 
	\\ 	\ \ =\ \ \ \ &\sum_{\aal=\0}^{(m-1)\1}\tau_{m,\1+\aal}\,\sum_{\bbe\colon|\bbe|=\aal}\,\int_{\bbe/2-\1/2}^{\nn-\1/2-\bbe/2}	&&\ \ =\ \ \ \  \sum_{\aal=\0}^{(m-1)\1}\tau_{m,\1+\aal}\,\sum_{\bbe\colon|\bbe|=\aal}\,\int_{-\1/2-\bbe/2}^{\nn-\1/2-\bbe/2}
	\label{eq:A_m,alt2} 
\end{alignat} 
is the integral approximation to the sum $\sum_{\kk=\0}^{\nn-\1}f(\kk)$, 
%Here and in the sequel we let, for brevity, 
%\begin{equation}\label{eq:int}
%	\int_\uu^\vv:=\int_\uu^\vv \dd \xx\,f(\xx),   
%\end{equation} 
\begin{equation}\label{eq:ga_j}
\ga_{m,\jj}:=\prod_{r=1}^p \ga_{m,j_r},\quad 
	\ga_{m,j}:=(-1)^{j-1}\,\frac2j\,\binom{2m}{m+j}\Big/ \binom{2m}{m}, 
\end{equation}
\begin{equation}\label{eq:tau_j}
\tau_{m,\jj}:=\prod_{r=1}^p \tau_{m,j_r},\quad
	\tau_{m,j}:=
	\sum_{\be=0}^{\lfloor m/2-j/2\rfloor}\ga_{m,j+2\be}
	\,=\sum_{\be=0}^\infty\ga_{m,j+2\be}, 
\end{equation}
and 
$R_m$ is the remainder given by the formula 
\begin{equation}\label{eq:R_m}
	R_m:=\frac{m}{2^{2m+p-1}}\,\sum_{\|\aal\|=2m}\frac1{\aal!}
\,\int_0^1\dd s\,(1-s)^{2m-1}\int_{-\1}^{\1}\dd \vv\,\vv^\aal
	\sum_{\jj=\1}^{m\1}\ga_{m,\jj}\,\jj^{\aal+\1} \sum_{\kk=\0}^{\nn-\1}f^{(\aal)}(\kk+s\jj\vv/2). 
\end{equation}
%\big(One may note here that, in each of the formulas \eqref{eq:A_m}, \eqref{eq:A_m,alt1}, and \eqref{eq:A_m,alt2}, the first expression is a linear combination of integrals over $p$-dimensional intervals centered at the point $(\nn-\1)/2$, whereas the endpoints of each of the intervals corresponding to the second expression differ by the vector $\nn$.\big)
The sum of all the coefficients of the integrals in each of the expressions \eqref{eq:A_m}, \eqref{eq:A_m,alt1}, and \eqref{eq:A_m,alt2} of $A_m$ is   
\begin{equation}\label{eq:sum ga_j}
	\sum_{\jj=\1}^{m\1}\ga_{m,\jj}\,\sum_{\iii=\0}^{\jj-\1}1=\sum_{\jj=\1}^{m\1}\ga_{m,\jj}\, \jj^{\1}
	=\sum_{\bbe=(1-m)\1}^{(m-1)\1}\tau_{m,\1+|\bbe|}=1. 
\end{equation}

If $M_{2m}$ is a real number such that 
\begin{equation}\label{eq:<M}
	\Big|\sum_{\kk=\0}^{\nn-\1}f^{(\aal)}(\kk+\uu)\Big|\le M_{2m}\quad\text{for all}\  
%IP06-02-17	
\aal \ \text{with}\ \|\aal\|=2m 
	\quad\text{and\quad all}\ \uu\in(-m\1/2,m\1/2], 
\end{equation}
then the remainder $R_m$ can be bounded as follows: 
\begin{align}
	|R_m|&\le 
	\frac{M_{2m}}{2^{2m}}\,\sum_{\|\aal\|=2m}\frac1{(\aal+\1)!}
%	\prod_{r=1}^p\sum_{j=1}^m|\ga_{m,j}|j^{\al_r+1} 
	\sum_{\jj=\1}^{m\1}|\ga_{m,\jj}|\,\jj^{\aal+\1} \label{eq:R<} \\ 
&\le M_{2m}\,\frac{1.0331(\pi m)^{(p+1)/2}}{(2m+1)!}\,(\ka pm)^{2m}, 
\label{eq:R<<}
\end{align}
where 
\begin{equation}\label{eq:ka}
	\ka:=\sqrt{\frac{\La_*}4}=0.27754\dots  
\end{equation}
and 
\begin{equation}\label{eq:La}
	\La_*:=\max_{0<t<1}\La(t)=0.3081\dots,\quad 
	\La(t):=(1-t)^{t-1} (1+t)^{-1-t} t^2. 
\end{equation} 
If $m\ge2$, then the factor $1.0331$ in \eqref{eq:R<<} can be replaced by $1.001$. 
\end{theorem}

Recall the convention that the sum of an empty family is $0$. 
In particular, if $\wedge\nn=0$, then %IP 
$\sum_{\kk=\0}^{\nn-\1}f(\kk)=0=A_m=R_m$. 

Also, it is clear that $R_m=0$ if the function $f$ is any polynomial of degree at most $2m-1$.  

One may note here that, in each of the formulas \eqref{eq:A_m}, \eqref{eq:A_m,alt1}, and \eqref{eq:A_m,alt2}, the first expression is a linear combination of integrals of the form $\int_{-\lla}^{\nn-\1+\lla}$ for some $\lla\in\R^p$ with $|\lla|\le(m-2)\1/2$. So, provided that $\nn\ge(m-1)\1$, each of these integrals equals the Lebesgue integral of the function $f$ over the $p$-dimensional  interval $[-\lla,\nn-\1+\lla]$, symmetric about the point $(\nn-\1)/2$. 

In contrast, the second expression in each of the formulas \eqref{eq:A_m}, \eqref{eq:A_m,alt1}, and \eqref{eq:A_m,alt2} is a linear combination of integrals of the form $\int_{\lla}^{\nn+\lla}$ for some $\lla\in\R^p$; so, each of these integrals equals the Lebesgue integral of the function $f$ over the $p$-dimensional  interval $[\lla,\nn+\lla]$, whose endpoints differ by the vector $\nn$. This observation holds whether the condition $\nn\ge(m-1)\1$ holds ot not.

\begin{remark}\label{rem:vect}
As in \cite{euler-maclaurin-alt} in the special case of ordinary sums, here, instead of assuming that the function $f$ is real-valued, one may assume, more generally, that $f$ takes values in any normed space. In particular, one may allow $f$ to take values in the $q$-dimensional complex space $\CC^q$, for any natural $q$. 
An advantage of dealing with a vector-valued function %such as the one defined by \eqref{eq:f,vect} 
(rather than separately with each of its coordinates) is that this way one has to compute the coefficients -- say $\tau_{m,\bbe}$ in \eqref{eq:A_m,alt2} %and $B_{2j}/(2j)!$ in \eqref{eq:A^EM} 
-- only once, for all the components of the vector function.  
\hfill\qed
\end{remark}

\section{Application to summing (possibly divergent) multi-index series}\label{series}

%The alternative summation formula presented in Theorem~\ref{th:} can be used for summing (possibly divergent) series, as follows.  

Let us say that a function $F$ on $\R^p$ is an antiderivative of the function $f$ if 
\begin{equation*}
	F^{(\1)}=f. 
\end{equation*}
Clearly, this notion is a generalization of the corresponding notion for functions on $\R$. It is also clear that an antiderivative exists and can be obtained by taking the iterated indefinite integral \break 
$\int\dd x_1\cdots\int\dd x_p\,f(x_1,\dots,x_p)$. 

%For $\xx=(x_1,\dots,x_p)\in\R^p$, let $\vee\xx:=x_1\vee\dots\vee x_p$ and $\wedge\xx:=x_1\wedge\dots\wedge x_p$. 

As usual, let $[p]:=\{1,\dots,p\}$. 
For each set $J\subseteq[p]$, let $|J|$ denote the cardinality of $J$, and also let 
$\1_J:=(\ii\{1\in J\},\dots,\ii\{p\in J\})$.   
%where $\ii\{A\}$ denotes the indicator of an assertion $A$. 
In particular, $\1_{[p]}=\1$ and $\1_{\emptyset}=\0$. 

The alternative summation formula presented in Theorem~\ref{th:} can be used for summing (possibly divergent) multi-index series, as follows.

\begin{theorem}\label{prop:series}
Let $m_0$ be a natural number, and suppose that $m\ge m_0$. 
Let $F$ be any antiderivative of $f$. 
Suppose that %a function $f\in C^{2m-}$ is such that  
\begin{equation}\label{eq:f^ to0}
	F^{(\aal)}(x)\underset{\vee\xx\to\infty}\longrightarrow0 \ 
	\ \text{for each\  $\aal\in\Z_+^p$\ with\  $\|\aal\|=2m_0$ }
\end{equation}
and the series 
\begin{equation}\label{eq:R unif}
	\sum_{\kk=\0}^{\infty\1} f^{(\aal)}(\kk+\uu) \ \text{converges uniformly in $\uu\in[-m\1/2,m\1/2]$\  
	for each\  $\aal\in\Z_+^p$\  with\  $\|\aal\|=2m$, }
\end{equation}
in the sense that $\sum_{\kk=\0}^{\nn-\1} f^{(\aal)}(\kk+\uu)$ converges uniformly as $\wedge\nn\to\infty$. 
%Let $F$ be any antiderivative of $f$. %, so that $F'=f$. 
Then 
\begin{equation}\label{eq:series}
\sum_{\kk\ge\0}^\Alt f(\kk):=\lim_{\wedge\nn\to\infty}	\Big(\sum_{\kk=\0}^{\nn-\1} f(\kk)-%A_{m_0,F}(\nn)+(-1)^p A^\emptyset_{m_0,F}(\0)
\tA_{m_0,F}(\nn)\Big)=
	%G_{m,F}
	(-1)^p A^\emptyset_{m,F}(\0)-R_{m,f}(\infty), 
\end{equation}
where (cf.\ \eqref{eq:A_m}, \eqref{eq:A_m,alt1},  and \eqref{eq:A_m,alt2})   
\begin{equation}\label{eq:tA_m,F(n)}
	\tA_{m,F}(\nn):=\sum_{\emptyset\ne J\subseteq[p]}(-1)^{p-|J|}A^J_{m,F}(\nn), 
\end{equation}
\begin{align}
A^J_{m,F}(\nn):=&\sum_{\jj=\1}^{m\1}\ga_{m,\jj} 
\sum_{\iii=\0}^{\jj-\1}F(\nn\1_J-\1+\jj/2-\iii) \label{eq:G} \\ 
	=&\sum_{\bbe=(1-m)\1}^{(m-1)\1}\tau_{m,\1+|\bbe|}\,F(\nn\1_J-\1/2-\bbe/2) 
	\label{eq:G_m,alt1} 
	\\ 
	=& \sum_{\aal=\0}^{(m-1)\1}\tau_{m,\1+\aal}\,\sum_{\bbe\colon|\bbe|=\aal}\,F(\nn\1_J-\1/2-\bbe/2),   
	\label{eq:G_m,alt2}
\end{align}
and (cf.\ \eqref{eq:R_m}) 
\begin{equation}\label{eq:R_m,f}
	R_{m,f}(\infty):=\frac{m}{2^{2m+p-1}}\,\sum_{\|\aal\|=2m}\frac1{\aal!}
\,\int_0^1\dd s\,(1-s)^{2m-1}\int_{-\1}^{\1}\dd \vv\,\vv^\aal
	\sum_{\jj=\1}^{m\1}\ga_{m,\jj}\,\jj^{\aal+\1} \sum_{\kk=\0}^{\infty\1}f^{(\aal)}(\kk+s\jj\vv/2). 
\end{equation} 

If condition \eqref{eq:<M} holds for all $\nn\in\Z_+^p$, then one can replace $R_m$ in \eqref{eq:R<}--\eqref{eq:R<<} by $R_{m,f}(\infty)$, so that 
\begin{equation}
	|R_{m,f}(\infty)|\le M_{2m}\,\frac{1.0331(\pi m)^{(p+1)/2}}{(2m+1)!}\,(\ka pm)^{2m}.  
\label{eq:R<<,infty}
\end{equation} 

Looking, say, at the expression of $A^J_{m,F}(\nn)$ in \eqref{eq:G_m,alt2}, one may note that 
\begin{equation}\label{eq:A^empty}
	A^\emptyset_{m,F}(\0)=A^\emptyset_{m,F}(\nn)=A^J_{m,F}(\0)=
	\sum_{\aal=\0}^{(m-1)\1}\tau_{m,\1+\aal}\,\sum_{\bbe\colon|\bbe|=\aal}\,F(\bbe/2-\1/2) 
\end{equation}
for all $\nn\in\Z_+^p$ and $J\subseteq[p]$. 
\end{theorem}

The limit $\sum_{\kk\ge\0}^\Alt f(\kk)$ in \eqref{eq:series} may be referred to as the (generalized) sum of the possibly divergent multi-index series $\sum_{\kk=\0}^{\infty\1} f(\kk)$ by means of the Alt formula \eqref{eq:}. 
%The constant term 
%$(-1)^p A^\emptyset_{m_0,F}(\0)$, not depending on $\nn$, is introduced under the limit sign in \eqref{eq:series} so that the definition of the generalized sum $\sum_{\kk\ge\0}^\Alt f(\kk)$ fully agree with the corresponding ``one-dimensional'' definition in \cite[(5.3)]{euler-maclaurin-alt}. 

Theorem~\ref{prop:series} is a multi-index extension of Proposition~5.1 in \cite{euler-maclaurin-alt}.

%Also, it is clear that $R_m=0$ if the function $f$ is any polynomial of degree at most $2m-1$. 

\bigskip 

To compute the generalized sum %s %in \eqref{eq:series} and \eqref{eq:seriesEM} 
$\sum_{\kk\ge\0}^\Alt f(\kk)$ %and $\sum_{k\ge0}^\EM f(k)$ 
effectively, one has to ensure that the remainder %s 
$R_{m,f}(\infty)$ %and $R^\EM_{m,f}(\infty)$ 
can be made arbitrarily small. This can be done as follows. 
%It can be seen that 
%the bounds in \eqref{eq:|R|<}--\eqref{eq:<M} and \eqref{eq:R^EM<} are rather tight. Therefore, usually the only way to ensure that the remainders $R_{m,f}(\infty)$ and $R^\EM_{m,f}(\infty)$ be small will be to make the high-order derivatives $f^{(2m)}$  and $f^{(2m-1)}$ small. This can be achieved by the following simple trick. 

For any function $h\colon\R^p\to\R$ and any $\cc\in\R^p$, let $h_\cc$ denote the $\cc$-shift of $h$ defined by the formula 
\begin{equation*} %IP
	h_\cc(\xx):=h(\xx+\cc)
\end{equation*}
for all $\xx\in\R^p$. 
Note that, if $F$ is an antiderivative of $f$, then $F_\cc$ is an antiderivative of $f_\cc$. 

\begin{theorem}\label{th:c}
Suppose that the conditions of Theorem~\ref{prop:series} hold. Take any $\cc\in\Z_+^p$. Then 
\begin{equation*}
	\sum_{\kk\ge\0}^\Alt f(\kk)=\sum_{\kk=\0}^{\cc-\1}f(\kk)-\tA_{m,F}(\cc)-R_{m,f,\cc}(\infty), 
\end{equation*}
where 
\begin{equation}\label{eq:R_c}
	R_{m,f,\cc}(\infty):=-\sum_{\emptyset\ne J\subseteq[p]}(-1)^{p-|J|}R_{m,f_{\cc\1_J}}(\infty)
\end{equation}
(cf.\ \eqref{eq:tA_m,F(n)}). 
\end{theorem}

%\begin{equation}\label{eq:tA_m,F(n)}
%	\tA_{m,F}(\nn):=\sum_{\emptyset\ne J\subseteq[p]}(-1)^{p-|J|}A^J_{m,F}(\nn), 
%\end{equation}

Under the conditions of Theorem~\ref{prop:series}, the remainder $R_{m,f,\cc}(\infty)$ can be made arbitrarily small by making  $\wedge\cc$ large enough. The price to pay for this will be the need to compute a possibly large partial sum $\sum_{\kk=\0}^{\cc-\1}f(\kk)$ of the series. 

Theorem~\ref{th:c} is a multi-index extension of Corollary~5.6 in \cite{euler-maclaurin-alt}.

%%%\begin{remark}\label{rem:polytopes}
%%%It should be possible to extend 
%%%Theorem~\ref{th:} to the case when the function $f$ is defined on a convex lattice polytope, using inclusion-exclusion types of formulas expressing the indicator function of a polytope as a linear combination (with coefficients of the form $(-1)^q$ with integer $q$) of the indicator functions of (finitely generated) cones; cf.\ e.g.\ 
%%%\cite{lawrence,KSW_duke07,agap-godin16}. 
%%%Actually, formulas \eqref{eq:FTC} and \eqref{eq:incl-excl}, used in the proofs of Theorems~\ref{prop:series} and \ref{th:c}, are simplest examples of such inclusion-exclusion formulas. 
%%%\end{remark}

%\newpage

\section{Application to sums over the integral points of integral lattice polytopes}\label{polytopes}

Let $P$ be an integral polytope in $\R^p$, that is, the convex hull of a finite subset of $\Z^p$. 
%IP10-29-17
Suppose that $P$ is of full dimension, $p$. 
Let $V$ denote the set of all vertices (that is, extreme points) of $P$. 
%Suppose that the polytope $P$ in $\R^p$ is simple, that is, exactly $p$ edges are incident with each of its vertices. 
%
%Recall that a matrix is called unimodular if its determinant is $1$ or $-1$. 

%By \cite[Theorem]{lawrence}, % and \cite[Theorem~(5.4)]{barv94%barv-pommer}, 

%By \cite[Theorem]{lawrence}, %one has %\eqref{eq:P decomp} with $I_\vv=\{1\}$ (say) for all $\vv\in V$
% and \cite[Theorem~(5.4)]{barv94%barv-pommer}, 
By the main result of Haase~\cite{haase05}, 
for each $\vv\in V$ there exist a finite set $I_\vv$, a map $I_\vv\ni i\mapsto t_{\vv,i}\in\{0,1\}$, 
a map $I_\vv\ni i\mapsto A_{\vv,i}$ into the set of all nonsingular $p\times p$ matrices over $\Z$, and a map $I_\vv\ni i\mapsto J_{\vv,i}$ into the set of all subsets of the set $[p]:=\{1,\dots,p\}$ such that     
\begin{equation}\label{eq:haase}
	%\ii\{\xx\in P\}
	\lb P\rb=\sum_{\vv\in V}\sum_{i\in I_\vv}(-1)^{t_{\vv,i}}\lb C_{\vv,i}\rb, 
%	\ii\{\xx\in\vv+A_{\vv,i}\R_+^p\}
\end{equation} 
where $\lb\cdot\rb$ denotes the indicator/characteristic function,  
\begin{equation}\label{eq:C_v,i}
	C_{\vv,i}:=\vv+A_{\vv,i}\R^+_{J_{\vv,i}}=\{\vv+A_{\vv,i}\xx\colon\xx\in\R^+_{J_{\vv,i}}\},   
\end{equation}
$$\R^+_J:=\prod_{j\in[p]}\R^+_{%IP10-29-17
1-\lb J\rb(j)}
\quad\text{for}\quad J\subseteq[p],$$ %$\vp_j:\lb J\rb(j)$, 
and 
\begin{equation*}
	\R^+_\vp:=
	\begin{cases}
	(0,\infty)&\text{ if }\vp=0, \\ 
	[0,\infty)&\text{ if }\vp=1 
	\end{cases}
\end{equation*} 
(so that the closure of $C_{\vv,i}$ is a polyhedral cone, for each pair $(\vv,i)$). 
In the case when the polytope $P$ is simple, decomposition \eqref{eq:haase} was obtained earlier by Lawrence~\cite{lawrence}. To extend Lawrence's result, Haase used virtual infinitesimal deformations of vertices of $P$, identified with regular triangulations of the normal cones at the vertices. 

\begin{proposition}\label{prop:C-decomp}
%\emph{\cite{sums-via-ints-multivar}}\quad 
Let $A$ be any nonsingular $p\times p$ matrix over $\Z$, and let $J$ be any subset of the set $[p]$. 
Then there exist a 
finite set $I$, a map $I\ni i\mapsto A_i$ into the set of all \emph{unimodular} $p\times p$ matrices over $\Z$, % with determinant in the set $\{-1,1\}$, 
and a map $I\ni i\mapsto J_i$ into the set of all subsets of the set $[p]$ such that    
\begin{equation}\label{eq:C-decomp}
	\lb A\R^+_J\rb=\sum_{i\in I}\lb A_i\R^+_{J_i}\rb.  
\end{equation}
(Recall that a matrix is called unimodular if its determinant is $1$ or $-1$.)
\end{proposition}

%The particular case of Proposition~\ref{prop:C-decomp} for simple polytopes was obtained in \cite{sums-via-ints-multivar}, using the mentioned Lawrence decomposition \cite{lawrence} rather than Haase's decomposition \eqref{eq:haase}. 
 
Thus, one can strengthen the statement on the decomposition \eqref{eq:haase} as follows: 

\begin{corollary}\label{cor:decomp}
One may assume that all the matrices $A_{\vv,i}$ in \eqref{eq:haase}--\eqref{eq:C_v,i} are unimodular.
\end{corollary}

A similar decomposition, but with %unspecified 
polyhedral cones of lower dimensions, was obtained in \cite{barv94}.

The following corollary is almost immediate from Theorem~\ref{th:} and Corollary~\ref{cor:decomp}. 

\begin{corollary}\label{cor:poly}
Suppose that 
%$P$ is simple, so that a decomposition of the form \eqref{eq:P decomp} takes place. 
%Suppose also that 
the function $f$ is compactly supported. Then 
\begin{equation}\label{eq:poly}
	\sum_{\kk\in P\cap\Z^p}f(\kk)
	=A_m(f,P)-R_m(f,P),
\end{equation}
where 
\begin{align}\label{eq:A_m(f,P)}
	A_m(f,P) :=\sum_{\bbe=(1-m)\1}^{(m-1)\1}\tau_{m,\1+|\bbe|}\,
	\sum_{\vv\in V}(-1)^{t_\vv}\sum_{i\in I_\vv}
\	\;\int\limits_{C_{\vv,i}+A_{\vv,i}(\1_{J_{\vv,i}}-(\1+\bbe)/2)}\dd\xx f(\xx)
\end{align} 
is the integral approximation to the sum $\sum\limits_{\kk\in P\cap\Z^p}f(\kk)$ 
and 
$R_m(f,P)$ is the remainder given by the formula 
\begin{align*}%\label{eq:R_m}
	&R_m(f,P):= \\ 
	&\frac{m}{2^{2m+p-1}}\,\sum_{\|\aal\|=2m}\frac1{\aal!}
\,\int_0^1\dd s\,(1-s)^{2m-1}\int_{-\1}^{\1}\dd \uu\,\uu^\aal
	\sum_{\jj=\1}^{m\1}\ga_{m,\jj}\,\jj^{\aal+\1} 
	\sum_{\vv\in V}\sum_{i\in I_\vv}(-1)^{t_\vv} 
	\sum_{\kk\ge\0}g_{\vv,i}^{(\aal)}(\kk+\1_{J_{\vv,i}}+s\jj\uu/2),  
\end{align*}
with 
\begin{equation*}
	g_{\vv,i}(\yy):=f(\vv+A_{\vv,i}\yy)
\end{equation*}
for $\yy\in\R^p$. 
If $M_{2m}$ is a real number such that 
\begin{equation*}%\label{eq:<M}
	\Big|\sum_{\vv\in V}\sum_{i\in I_\vv}(-1)^{t_\vv} 
	\sum_{\kk\ge\0}g_{\vv,i}^{(\aal)}(\kk+\uu)\Big|\le M_{2m}\quad\text{whenever}\quad  
%IP06-02-17	 
\|\aal\|=2m 
	\quad\text{and}\quad |\uu|\le(\tfrac m2+1)\1, 
\end{equation*}
then 
\begin{align*}
	|R_m(f,P)|  
	\le 
M_{2m}\,\frac{1.0331(\pi m)^{(p+1)/2}}{(2m+1)!}\,(\ka pm)^{2m}, 
%\label{eq:R<<}
\end{align*}
where $\ka$ is as in \eqref{eq:ka}. 
\end{corollary}

Indeed, for $J\subseteq[p]$, let 
\begin{equation*}
	\Z^+_J:=\Z^p\cap\R^+_J=\Z_+^p+\1_J,  
\end{equation*}
where $\Z_+:=\Z\cap[0,\infty)$. 
Note that $A\Z^p=\Z^p$ for any unimodular matrix $A$ over $\Z$. 
Now Corollary~\ref{cor:poly} follows by 
Corollary~\ref{cor:decomp} and Theorem~\ref{th:} because 
\begin{equation*}%\label{eq:sum f=sum g}
		\sum_{\kk\in C_{\vv,i}\cap\Z^p}f(\kk)
		=\sum_{\elll\in\Z^+_{J_{\vv,i}}}f(\vv+A_{\vv,i}\elll)
		=\sum_{\elll\ge\0}f(\vv+A_{\vv,i}(\elll+\1_{J_{\vv,i}}))
		=\sum_{\elll\ge\0}g_{\vv,i}(\elll+\1_{J_{\vv,i}})
\end{equation*}
and 
\begin{equation*}
	\int_{[-\1/2-\bbe/2,\,\infty\1)}\dd\yy\,g_{\vv,i}(\yy+\1_{J_{\vv,i}})
	=\int_{C_{\vv,i}+A_{\vv,i}(\1_{J_{\vv,i}}-(\1+\bbe)/2)}\dd\xx f(\xx).  
\end{equation*}
%the first equality in \eqref{eq:sum f=sum g} is based on the simple observation that $A\Z^p=\Z^p$ for any unimodular matrix $A$.  

%
The expression for $A_m(f,P)$ in \eqref{eq:A_m(f,P)} is based on the second expression for $A_m$ in \eqref{eq:A_m,alt1}; of course, one can quite similarly use any one of the other 5 expressions in \eqref{eq:A_m}--\eqref{eq:A_m,alt2}. 

%Instead of the mentioned \cite[Theorem~(5.4)]{barv94}, one can similarly use the simpler but oftentimes less powerful decomposition 
%\begin{equation*}
%	\Z^p\cap A\R_+^p=(\Z^p\cap A[0,1)^p)+A\Z_+^p=\bigcup_{\kk\in\Z^p\cap A[0,1)^p}(\kk+A\Z_+^p)     
%\end{equation*}
%for any nonsingular $p\times p$ matrix $A$ over $\Z$,   
%where the sets $\kk+A\Z_+^p$ under the union sign are pairwise disjoint for distinct $\kk\in\Z^p\cap A[0,1)^p$, so that 
%\begin{equation*}
%	\lb\Z^p\cap A\R_+^p\rb=\sum_{\kk\in\Z^p\cap A[0,1)^p}\lb\kk+A\Z_+^p\rb;      
%\end{equation*} 
%see e.g.\ \cite[Example~3.3]{barv-pommer}. 

Notable differences between Corollary~\ref{cor:poly} and the main result of \cite{KSW_duke07} (Theorem~2 therein) include the following: (i) in \cite[Theorem~2]{KSW_duke07}, the summation is over all faces of the polytope $P$, whereas in \eqref{eq:A_m(f,P)} the corresponding summation is only over the vertices of $P$ and (ii) instead of the plain summation $\sum_{\kk\in P\cap\Z^p}f(\kk)$ in \eqref{eq:poly}, in the corresponding sum in \cite{KSW_duke07} the summands $f(\kk)$ are weighted (in accordance with the dimension of the relative interior of the face given that $\kk$ belongs to 
that relative interior). 

Note also that \cite[Theorem~2]{KSW_duke07} is obtained for simple polytopes. In \cite{agap-godinho07}, this result was extended to allow more general weights, and then further generalized to non-simple polytopes in \cite{agap-godin16}. 

It should be possible to extend Corollary~\ref{cor:poly} to the case when the function $f$ is a so-called symbol in the sense of H\"ormander \cite{hoermander} %; cf.\ \cite[Theorem~3]{KSW_duke07}.
%
%Extend Corollary~\ref{cor:poly} to the case when the function $f$ is a so-called symbol in the sense of H\"ormander \cite{hoermander} 
-- cf.\ \cite[Theorem~3]{KSW_duke07}, as well as conditions \eqref{eq:f^ to0} and \eqref{eq:R unif}. 
\big(Recall that a function $f\in C^\infty(\R^p)$ is called a symbol of
order $N$ if for every  $\aal\in\Z_+^p$ there is a real constant $C_\aal$ such that $|f^{(\aal)}(\xx)|\le C_\aal(1+\|\xx\|)^{N-\|\aal\|}$ for all $\xx\in\R^p$; here, as before, $\|\cdot\|:=\|\cdot\|_1$.\big) 
One way to attack this goal could be to show that, for any $\aal\in\Z_+^p$ such that $\aal\le(m-1)\1$, the essential support (except possibly for a set of Lebesgue measure $0$) of the function 
\begin{equation}
	\sum_{\bbe\colon|\bbe|=\aal}\;\sum_{\vv\in V}\sum_{i\in I_\vv}(-1)^{t_{\vv,i}}
\	\lb C_{\vv,i}+A_{\vv,i}(\1_{J_{\vv,i}}-(\1+\bbe)/2)\rb 
\end{equation}
is bounded, presumably being just a perturbed version of the indicator of the polytope $P$; cf.\ \eqref{eq:A_m(f,P)} and the equality in \cite[formula~(89)]{KSW_duke07}. 

%Also, using triangulation, it should be possible to get rid of the condition in Corollary~\ref{cor:poly} that the polytope $P$ is simple. 
%; the price to pay for that will be the need to also include polytopes of dimensions $<p$ into the decomposition. 

Moreover, in view of the results of Section~\ref{series}, 
it appears not unlikely that Corollary~\ref{cor:poly} could be extended to general polyhedral sets. 

%\appendix

\section{Proofs% of Theorem~\ref{th:}
}\label{proof}
\begin{normalsize}
\begin{proof}[Proof of Theorem~\ref{th:}]
Take any $\kk$ (in $\Z_+^p$) such that $\kk\le\nn-\1$ 
and consider the Taylor expansion 
\begin{equation}\label{eq:taylor}
	f(\xx)=\sum_{\|\aal\|\le2m-1}\frac{f^{(\aal)}(\kk)}{\aal!}\,\uu^\aal
	+\sum_{\|\aal\|=2m}\frac{2m}{\aal!}\,\uu^\aal\,\int_0^1\dd s\,(1-s)^{2m-1}f^{(\aal)}(\kk+s\uu)
\end{equation}
for all $\xx\in(\kk-m\1/2,\,\kk+m\1/2]$, where $\uu:=\xx-\kk$. 
Integrating both sides of this identity in \break 
$\xx\in(\kk-\jj/2,\,\kk+\jj/2]$ (or, equivalently, in $\uu\in(-\jj/2,\,\jj/2]$) for each $\jj$ (in $\Z_+^p$) such that $\jj\le m\1$, then multiplying by $\ga_{m,\jj}$, and then summing in $\jj$, one has 
\begin{equation}\label{eq:A=S+R}
	A_{m,\kk}=S_{m,\kk}+R_{m,\kk},
\end{equation}
where 
\begin{align}
	A_{m,\kk}&:=\sum_{\jj=\1}^{m\1}\ga_{m,\jj}\int_{\kk-\jj/2}^{\kk+\jj/2}\dd \xx\,f(\xx), \label{eq:A_mk} \\ 
	S_{m,\kk}&:=\sum_{\|\aal\|\le m-1}\frac{f^{(2\al)}(\kk)}{(2\aal+1)!\,2^{2\|\aal\|}}
	\,\sum_{\jj=\1}^{m\1}\ga_{m,\jj}\, \jj^{2\aal+\1}, 
	\label{eq:S_mk}\\ 
	R_{m,\kk}&:=\sum_{\|\aal\|=2m}\frac{2m}{\aal!}\,
	\int_0^1\dd s\,(1-s)^{2m-1} \sum_{\jj=\1}^{m\1}\ga_{m,\jj} 
	\int_{-\jj/2}^{\jj/2}\dd \uu\;\uu^\aal\, f^{(\aal)}(\kk+s\uu) \notag \\ 
	&=\sum_{\|\aal\|=2m}\frac{2m}{\aal!}\,
	\int_0^1\dd s\,(1-s)^{2m-1} \sum_{\jj=\1}^{m\1}\ga_{m,\jj}\, 
	(\jj/2)^{\aal+\1}
	\int_{-\1}^{\1}\dd \vv\;\vv^\aal\, f^{(\aal)}(\kk+s\jj\vv/2);  \label{eq:R_mk}
\end{align}
the latter equality is obtained by the change of variables $\uu=\jj\vv$.

%	\frac1{(2m-1)!}\,\int_0^1\dd s\,(1-s)^{2m-1}
%	\sum_{j=1}^m\ga_{m,j} \int_{-j/2}^{j/2}\dd u\,u^{2m} f^{(2m)}(k+su) \notag \\ 
%	&\,=\frac1{(2m-1)!\,2^{2m+1}}\,\int_0^1\dd s\,(1-s)^{2m-1}\int_{-1}^{1}\dd v\,v^{2m}
%	\sum_{j=1}^m\ga_{m,j}j^{2m+1} f^{(2m)}(k+jsv/2). \notag 

%As usual, let $\de_\xx$ (respectively, $\de_x$) denote the Dirac measure at the point $\xx\in\R^p$ (respectively, at the point $x\in\R$). 
As noted before, in the special case $p=1$ Theorem~\ref{th:} turns into Theorem~3.1 of \cite{euler-maclaurin-alt}. 
So, without loss of generality (w.l.o.g.) $p\ge2$. 
Write 
\begin{equation}\label{eq:write} %IP
\sum_{\kk=\0}^{\nn-\1}\int_{\kk-\jj/2}^{\kk+\jj/2}\dd\xx\, f(\xx)
=\sum_{k_1=0}^{n_1-1}\cdots\sum_{k_p=0}^{n_p-1}
\int_{k_p-j_p/2}^{k_p+j_p/2}\dd x_p\cdots \int_{k_1-j_1/2}^{k_1+j_1/2}\dd x_1\,f(\xx). 
\end{equation}
In view of the multi-line display next after formula (9.7) in \cite{euler-maclaurin-alt} (note, in particular, the penultimate expression there), the right-hand side of \eqref{eq:write} can be rewritten as 
\begin{align*}
	&\sum_{k_1=0}^{n_1-1}\,\cdots\,\sum_{k_{p-1}=0}^{n_{p-1}-1}
	\,\sum_{i_p=0}^{j_p-1} \,\int_{i_p-j_p/2}^{n_p-1+j_p/2-i_p}\dd x_p
\int_{k_{p-1}-j_{p-1}/2}^{k_{p-1}+j_{p-1}/2}\dd x_{p-1}\cdots \int_{k_1-j_1/2}^{k_1+j_1/2}\dd x_1\,f(\xx) \\
=&\sum_{i_p=0}^{j_p-1} \int_{i_p-j_p/2}^{n_p-1+j_p/2-i_p}\dd x_p \,
\sum_{k_1=0}^{n_1-1}\cdots\sum_{k_{p-1}=0}^{n_{p-1}-1}
\, \int_{k_{p-1}-j_{p-1}/2}^{k_{p-1}+j_{p-1}/2}\dd x_{p-1}\cdots \int_{k_1-j_1/2}^{k_1+j_1/2}\dd x_1\,f(\xx) \\
&\vdots \\
=&\sum_{i_p=0}^{j_p-1} \int_{i_p-j_p/2}^{n_p-1+j_p/2-i_p}\dd x_p \cdots
\sum_{i_1=0}^{j_1-1} \int_{i_1-j_1/2}^{n_1-1+j_1/2-i_1}\dd x_1 \,f(\xx). 
\end{align*}
So, 
\begin{equation}
	\sum_{\kk=\0}^{\nn-\1}\int_{\kk-\jj/2}^{\kk+\jj/2}\dd\xx\, f(\xx)
	=\sum_{\iii=\0}^{\jj-\1} \int_{\iii-\jj/2}^{\nn-\1+\jj/2-\iii}\dd\xx \,f(\xx)  
\end{equation}
and hence, by \eqref{eq:A_mk},  
%
% 
%\begin{equation}\label{eq:==}
%	\sum_{k=0}^{n-1}\int_{(k-j/2,\,k+j/2]}\dd\de_{x}
%=\sum_{i=0}^{j-1}\int_{(i-j/2,\,n-1+j/2-i]}\dd\de_{x}
%=\sum_{i=0}^{j-1}\int_{(-1+j/2-i,\,n-1+j/2-i]}\dd\de_{x} 
%\end{equation}
%for any $n\in\Z_+$, any natural $j$, and any real $x$ -- cf.\ the multi-line display next after formula (9.7) in \cite{euler-maclaurin-alt} and the reasoning right after formula (9.16) therein. 
%Using the first equality in \eqref{eq:==}, we have 
%\begin{equation}\label{eq:de1}
%\begin{aligned}
%	\sum_{\kk=\0}^{\nn-\1}\int_{(\kk-\jj/2,\,\kk+\jj/2]}\dd\de_\xx
%	&=\sum_{k_1=0}^{n_1-1}\dots\sum_{k_p=0}^{n_p-1}\,\prod_{r=1}^p\int_{(k_r-j_r/2,\,k_r+j_/2]}\dd\de_{x_r}  \\ 
%	&=\prod_{r=1}^p\sum_{k_r=0}^{n_r-1}\int_{(k_r-j_r/2,\,k_r+j_r/2]}\dd\de_{x_r}  \\ 
%	&=\prod_{r=1}^p\sum_{i_r=0}^{j_r-1}\int_{(i_r-j_r/2,\,n_r-1+j_r/2-i_r]}\dd\de_{x_r} %\label{eq:1-dim} 
%	\\ 
%	&=\sum_{i_1=0}^{j_1-1}\dots\sum_{i_p=0}^{j_p-1}
%	\prod_{r=1}^p\int_{(i_r-j_r/2,\,n_r-1+j_r/2-i_r]}\dd\de_{x_r}  \\ 
%	&=\sum_{\iii=\0}^{\jj-1}\int_{(\iii-\jj/2,\,\nn-\1+\jj/2-\iii]}\dd\de_\xx.   
%\end{aligned}	
%\end{equation}
%Similarly, but using the last expression in \eqref{eq:==} instead of the penultimate one, we have 
%\begin{equation}\label{eq:de2}
%	\sum_{\kk=\0}^{\nn-\1}\int_{(\kk-\jj/2,\,\kk+\jj/2]}\dd\de_\xx
%	=\sum_{\iii=\0}^{\jj-1}\int_{(-\1+\jj/2-\iii,\,\nn-\1+\jj/2-\iii]}\dd\de_\xx. 
%\end{equation}
%
%Note also that $\int_{(\uu,\vv]}\dd \xx\,f(\xx)=\int_{\R^p}\dd \xx\,f(\xx)\,\int_{(\uu,\vv]}\dd \de_\xx$.  
%So, by \eqref{eq:A_mk}, multi-line display \eqref{eq:de1}, and \eqref{eq:A_m}, 
\begin{equation}\label{eq:sum A_mk=}
\begin{aligned}
	\sum_{\kk=\0}^{\nn-\1}A_{m,\kk}
	=\sum_{\jj=\1}^{m\1}\ga_{m,\jj}\sum_{\kk=\0}^{\nn-\1}\int_{\kk-\jj/2}^{\kk+\jj/2}\dd \xx\,f(\xx)
=\sum_{\jj=\1}^{m\1}\ga_{m,\jj} \sum_{\iii=\0}^{\jj-1} 
		\int_{\iii-\jj/2}^{\nn-\1+\jj/2-\iii}\dd \xx\,f(\xx)=A_m.
\end{aligned}		 
\end{equation}
Similarly, but using %IP \eqref{eq:de2} instead of \eqref{eq:de1}, 
the last expression in 
the mentioned multi-line display next after formula (9.7) in \cite{euler-maclaurin-alt} rather than the penultimate expression there, 
we have 
\begin{equation}\label{eq:sum A_mk==}
	\sum_{\kk=\0}^{\nn-\1}A_{m,\kk}=\sum_{\jj=\1}^{m\1}\ga_{m,\jj} \sum_{\iii=\0}^{\jj-1} 
		\int_{-\1+\jj/2-\iii}^{\nn-\1+\jj/2-\iii}\dd \xx\,f(\xx).  
\end{equation}
In particular, it follows that the two double sums in \eqref{eq:A_m} are the same.  

Suppose now that some $\iii$ and $\jj$ in $\Z_+^p$ and some $\bbe\in\Z^p$ are related by the condition $\bbe=2\iii-\jj+\1$. Then the condition $\1\le\jj\le m\1\ \&\ \0\le\iii\le\jj-\1$ is equivalent to the condition 
$$(1-m)\1\le\bbe\le(m-1)\1\ \&\ \1+|\bbe|\le\jj\le m\1\ \&\ (\jj-\1-|\bbe|)/2\in\Z_+^p.$$ 
So, 
\begin{equation}\label{eq:ttau1}
	\sum_{\jj=\1}^{m\1}\ga_{m,\jj}\sum_{\iii=\0}^{\jj-\1}\int_{\iii-\jj/2}^{\nn-\1+\jj/2-\iii}
	=\sum_{\bbe=(1-m)\1}^{(m-1)\1}\tilde\tau_{m,\1+|\bbe|}\,\int_{\bbe/2-\1/2}^{\nn-\1/2-\bbe/2}  
\end{equation}
and 
\begin{equation}\label{eq:ttau2}
	\sum_{\jj=\1}^{m\1}\ga_{m,\jj}\sum_{\iii=\0}^{\jj-\1}\int_{-\1+\jj/2-\iii}^{\nn-\1+\jj/2-\iii}
	=\sum_{\bbe=(1-m)\1}^{(m-1)\1}\tilde\tau_{m,\1+|\bbe|}\,\int_{-\1/2-\bbe/2}^{\nn-\1/2-\bbe/2},   
\end{equation}
where 
\begin{align*}
	\tilde\tau_{m,\1+|\bbe|}&:=
	\sum_{\jj=\1+|\bbe|}^{m\1}\ga_{m,\jj}\,\ii\big\{(\jj-\1-|\bbe|)/2\in\Z_+^p\big\} 
	\\
	&=\sum_{j_1=1+|\be_1|}^m\dots\sum_{j_p=1+|\be_p|}^m
	\,\prod_{r=1}^p\big(\ga_{m,j_r}\,\ii\big\{(j_r-1-|\be_r|)/2\in\Z_+\big\}\big) \\ 
	&=\prod_{r=1}^p\,\sum_{j_r=1+|\be_r|}^m
	\big(\ga_{m,j_r}\,\ii\big\{(j_r-1-|\be_r|)/2\in\Z_+\big\}\big) \\ 
	&=\prod_{r=1}^p\tau_{m,1+|\be_r|}=\tau_{m,\1+|\bbe|},  
\end{align*}
in view of \eqref{eq:ga_j} and \eqref{eq:tau_j}. 
%Thus, the equality in \eqref{eq:A_m,alt1} follows. 
%Here and elsewhere, %we use the convention $0^0:=1$, and 
%$\ii\{\cdot\}$ denotes the indicator function. 
Thus, by \eqref{eq:ttau1} and \eqref{eq:ttau2}, the first double sum in \eqref{eq:A_m} equals the first sum in \eqref{eq:A_m,alt1}, and the second double sum in \eqref{eq:A_m} equals the second sum in \eqref{eq:A_m,alt1}. 

Also, it is obvious that the first sum in \eqref{eq:A_m,alt2} equals the first sum in \eqref{eq:A_m,alt1}, and the second sum in \eqref{eq:A_m,alt2} equals the second sum in \eqref{eq:A_m,alt1}. 
%
%equality in \eqref{eq:A_m,alt1} follows from \eqref{eq:ttau}.

%The equality in \eqref{eq:A_m,alt2} is obvious. 

Next, for any $\aal$ (in $\Z_+^p$) with $\|\aal\|\le m-1$, 
\begin{equation}\label{eq:sum ga jj}
	\sum_{\jj=\1}^{m\1}\ga_{m,\jj}\, \jj^{2\aal+\1}
	=\sum_{j_1=1}^m\dots\sum_{j_p=1}^m\,\prod_{r=1}^p\big(\ga_{m,j_r}j_r^{2\al_r+1}\big)
	=\prod_{r=1}^p\sum_{j=1}^m \ga_{m,j}j^{2\al_r+1}=\ii\{\aal=\0\} 
\end{equation}
by formula (9.6) in \cite{euler-maclaurin-alt}. So, by \eqref{eq:S_mk}, 
\begin{equation}\label{eq:S=f}
	S_{m,\kk}=f(\kk).  
\end{equation}

Also, the case $\aal=\0$ in \eqref{eq:sum ga jj} shows that 
the first two sums in 
\eqref{eq:sum ga_j}, involving the $\ga_{m,\jj}$'s, are equal to $1$.  
The second equality 
in \eqref{eq:sum ga_j} follows from 
%, say, yet to be proved 
the equality of the first sums in \eqref{eq:A_m} and \eqref{eq:A_m,alt1} to each other
by taking there $\nn=m\1$ and 
$f(\xx)\equiv\ii\{(m/2-1)\1\le\xx\le m\1/2\}$; then each of the integrals in \eqref{eq:A_m}--\eqref{eq:A_m,alt2} equals $1$. 

%The equality in \eqref{eq:R_mk} is obtained by the easy change of variables $\uu=\jj\vv$. 
By \eqref{eq:R_mk} and \eqref{eq:R_m},  
\begin{equation*}
	\sum_{\kk=\0}^{\nn-\1}R_{m,\kk}=R_m. 
\end{equation*}
So, \eqref{eq:} follows immediately from \eqref{eq:A=S+R}, \eqref{eq:sum A_mk=}, and \eqref{eq:S=f}. 

In view of \eqref{eq:R_m} and \eqref{eq:<M}, 
\begin{equation*}
	|R_m|\le \tilde R_m:=
	M_{2m}\,\frac{m}{2^{2m+p-1}}\,\sum_{\|\aal\|=2m}\frac1{\aal!}
\,\int_0^1\dd s\,(1-s)^{2m-1}\int_{-\1}^{\1}\dd \vv\,|\vv|^\aal
	\sum_{\jj=\1}^{m\1}|\ga_{m,\jj}|\,\jj^{\aal+\1}.
\end{equation*}
Computing the integrals here, it is easy to check that $\tilde R_m$ equals the upper bound in \eqref{eq:R<}. 
On the other hand, using the multinomial formula, the definition of $\ga_{m,\jj}$ in \eqref{eq:ga_j}, and the H\"older inequality $\Big(\sum_{r=1}^p |v_r j_r|\Big)^{2m}\le p^{2m-1}\sum_{r=1}^p |v_r j_r|^{2m}$, we see that 
\begin{equation}\label{eq:tR}
\begin{aligned}
	\tilde R_m=&\frac{M_{2m}}{2^{2m}(2m)!}\,\sum_{\jj=\1}^{m\1}|\ga_{m,\jj}|\,\jj^\1
	\int_{\0}^{\1}\dd \vv\,\sum_{\|\aal\|=2m}\frac{(2m)!}{\aal!}\,(\vv\jj)^\aal \\ 
=& 	\frac{M_{2m}}{2^{2m}(2m)!}\,\sum_{\jj=\1}^{m\1}|\ga_{m,\jj}|\,\jj^\1
	\int_{\0}^{\1}\dd \vv\,\Big(\sum_{r=1}^p v_r j_r\Big)^{2m} \\ 
\le& \frac{M_{2m}p^{2m-1}}{2^{2m}(2m)!}\,
\Big(\sum_{j_1=1}^m\dots\sum_{j_p=1}^m\,|\ga_{m,j_1}|j_1\dots|\ga_{m,j_p}|j_p\Big) \sum_{r=1}^p j_r^{2m}
	\int_{\0}^{\1}v_r^{2m}\dd \vv\ \\ 	
=& \frac{M_{2m}p^{2m}}{2^{2m}(2m+1)!}\,
\sum_{j=1}^m|\ga_{m,j}|j^{2m+1}\, \Big(\sum_{j=1}^m|\ga_{m,j}|j\Big)^{p-1}. 	
\end{aligned}
\end{equation}
By Proposition 4.4 in \cite{euler-maclaurin-alt}, 
\begin{equation}\label{eq:sum gaj}
	\sum_{j=1}^m|\ga_{m,j}|j^{2m+1}\le1.0331\pi\La_*^m m^{2m+1},  
\end{equation}
%IP in fact, this inequality is trivial for $m=1$, % \1.033vs1.001.nb 
and for $m\ge2$ the factor $1.0331$ can be replaced by $1.001$. 

It follows from \cite{watson} that $\Ga(x+1)/\Ga(x+1/2)>\sqrt{x+1/\pi}$ for real $x>0$. For $x=m\in\N$, this inequality can be rewritten as $2^{2m}\Big/\binom{2m}{m}<\sqrt{\pi m+1}$. 
So, in view of \eqref{eq:ga_j}, % and Lemma~\ref{lem:sum ga} below, 
\begin{equation}\label{eq:<sqrt}
	\sum_{j=1}^m|\ga_{m,j}|j=2^{2m}\Big/\binom{2m}{m}-1<\sqrt{\pi m}.  
\end{equation}
Collecting \eqref{eq:tR}, \eqref{eq:sum gaj}, \eqref{eq:<sqrt}, and \eqref{eq:ka}, we obtain \eqref{eq:R<<}. 

Theorem~\ref{th:} is now completely proved. %, modulo Lemma~\ref{lem:sum ga}. 
\end{proof}

%%%\begin{lemma}\label{lem:sum ga}
%%%\begin{equation}
%%%	2^{2m}\Big/\binom{2m}{m}<\sqrt{\pi m+1}. 
%%%\end{equation}
%%%\end{lemma}
%%%
%%%\begin{proof} % \multivar\binom-lemma.nb
%%%Let $r(m):=\binom{2m}{m}\sqrt{\pi m+1}/2^{2m}$ and 
%%%$r_1(m):=\frac{r(m+1)}{r(m)}=\frac{n}{d}$, where $n:=(2 m+1)\sqrt{\pi  m+\pi +1}$ and $d:=2 (m+1)\sqrt{\pi  m+1}$. Then $d^2 - n^2=3 - \pi + m (4 - \pi)\ge3 - \pi + (4 - \pi)>0$. 
%%%So, $r_1(m)<1$ for $m\ge1$. Therefore, $r(m)$ is decreasing in natural $m$. Also, by Stirling's formula, $r(m)\to1$ as $m\to\infty$. So, $r(m)>1$ for all natural $m$, which completes the proof of Lemma~\ref{lem:sum ga}. 
%%%%	\frac{2 m+1}{2 (m+1) }\,\sqrt{\frac{\pi  m+\pi +1}{\pi  m+1}}.  
%%%%\begin{equation}
%%%%	r(m):=\binom{2m}{m}\sqrt{\pi m+1}/2^{2m}\quad\text{and}\quad 
%%%%	r_1(m):=\frac{r(m+1)}{r(m)}=\frac{n}{d}, \quad\text{where}\quad 
%%%%	\frac{2 m+1}{2 (m+1) }\,\sqrt{\frac{\pi  m+\pi +1}{\pi  m+1}}. 
%%%%\end{equation}
%%%%Then $r_2(m):=r_1'(m)\,4 (m+1)^2 (\pi  m+1)^{3/2}\, \sqrt{\pi  m+\pi +1}=(4-\pi) \pi  m-(\pi -2) \pi +2$ is increasing in $m$, and $r_2(1)>0$. So, $r_2(m)>0$ for $m\ge1$, $r_1(m)$ is increasing to $1$ in $m\ge1$, whence $r(m)$ is decreasing in natural $m$. Also, by Stirling's formula, $r(m)\to1$ as $m\to\infty$. So, $r(m)>1$ for all natural $m$, which completes the proof of Lemma~\ref{lem:sum ga}. 
%%%\end{proof}

To prove Theorem~\ref{prop:series}, we shall need the following multidimensional generalization of the fundamental theorem of calculus (FTC). It is easy to prove and probably well known; however, I have not been able to find an appropriate reference. 

\begin{lemma}\label{lem:FTC}
\emph{(Multidimensional FTC)}\quad
Let $F$ be any antiderivative of $f$. Take any $\uu$ and $\vv$ in $\R^p$. %IP  such that $\uu\le\vv$. 
Then 
\begin{equation}\label{eq:FTC}
	\int_\uu^\vv\dd\xx f(\xx)=\sum_{J\subseteq[p]}(-1)^{p-|J|}F(\vv_J), 
\end{equation}
where $\vv_J:=\uu\1_{[p]\setminus J}+\vv\1_J=\uu+(\vv-\uu)\1_J$. 
\end{lemma}

\begin{proof}
This will be done by induction in $p$. For $p=1$, \eqref{eq:FTC} is the usual, one-dimensional FTC. Suppose that $p\ge2$ and that \eqref{eq:FTC} holds with %any $r\in\{1,\dots,p-1\}$ 
$p-1$ in place of $p$. 

Introduce some notation, as follows. For $\xx=(x_1,\dots,x_{p-1},x_p)\in\R^p$, let $\tilde\xx:=(x_1,\dots,x_{p-1})$, and similarly define $\tilde\uu$ and $\tilde\vv$. Also, for any $J\subseteq[p-1]$, define $\tilde\vv_J$ similarly to $\vv_J$, but based on $\tilde\uu$ and $\tilde\vv$ rather than on $\uu$ and $\vv$. For any function $h\colon\R^p\to\R$ and any real $x_p$, let $h_{x_p}$ denote the ``cross-section'' function from $\R^{p-1}$ to $\R$ defined by the formula $h_{x_p}(\tilde\xx):=h(\xx)$, again for $\xx=(x_1,\dots,x_{p-1},x_p)\in\R^p$. 
Note that, for each real $x_p$, the function $\big(F^{(\1_{\{p\}})}\big)_{x_p}$ is an antiderivative of the function $f_{x_p}$. 

For real $u$ and $v$, let $\De_{u,v}:=\de_v-\de_u$, where $\de_x$ is the Dirac measure at $x$. 
Consider the signed product measures 
\begin{equation*}
\De_{\uu,\vv}:=\De_{u_1,v_1}\otimes\dots\otimes\De_{u_p,v_p}=\sum_{J\subseteq[p]}(-1)^{p-|J|}\de_{\vv_J} 	
\end{equation*}
and $\tilde\De_{\uu,\vv}:=\De_{u_1,v_1}\otimes\dots\otimes\De_{u_{p-1},v_{p-1}}$, so that $\De_{\uu,\vv}=\tilde\De_{\uu,\vv}\otimes\De_{u_p,v_p}$. 

Now, appropriately rewriting the right-hand side of \eqref{eq:FTC} and then using the Fubini theorem and the induction hypothesis, we have %IP
\begin{alignat*}{2}
\sum_{J\subseteq[p]}(-1)^{p-|J|}F(\vv_J)&=\int_{\R^p}\dd\De_{\uu,\vv}\;F 
&&\text{\quad(rewriting)}\\ 
&=\int_\R\De_{u_p,v_p}(\dd x_p)\,\int_{\R^{p-1}}\dd\tilde\De_{\uu,\vv}\;F_{x_p} &&\text{\quad(Fubini)}\\ 
&=\int_\R\De_{u_p,v_p}(\dd x_p)\,\sum_{J\subseteq[p-1]}(-1)^{p-1-|J|}F_{x_p}(\tilde\vv_J) &&\text{\quad(similar rewriting)}\\ 
&=\int_{u_p}^{v_p}\dd x_p\;\sum_{J\subseteq[p-1]}(-1)^{p-1-|J|}
\frac{\partial}{\partial x_p} F_{x_p}(\tilde\vv_J) &&\text{\quad(one-dimensional FTC)}\\ 
%&=\int_{u_p}^{v_p}\dd x_p\;\Big(\sum_{J\subseteq[p-1]}(-1)^{p-1-|J|}F_{x_p}(\tilde\vv_J)\Big)^{(\1_{\{p\}})} &&\text{\quad(one-dimensional FTC)}\\ 
&=\int_{u_p}^{v_p}\dd x_p\;\sum_{J\subseteq[p-1]}(-1)^{p-1-|J|}\big(F^{(\1_{\{p\}})}\big)_{x_p}(\tilde\vv_J) && %\text{\quad(linearity of differentiation)}
\\  
&=\int_{u_p}^{v_p}\dd x_p\;\int_{\tilde\uu}^{\tilde\vv}\dd\tilde\xx\;f_{x_p}(\tilde\xx) 
&& \text{\quad(induction)} \\ 
&=	\int_\uu^\vv\dd\xx f(\xx). &&%\text{\quad(Fubini)}
\end{alignat*}
%\begin{align*}
%\sum_{J\subseteq[p]}(-1)^{p-|J|}F(\vv_J)&=\int_{\R^p}\dd\De_{\uu,\vv}\;F \\ 
%&=\int_\R\De_{u_1,v_1}(\dd x_1)\,\int_{\R^{p-1}}\dd\tilde\De_{\uu,\vv}\;F_{x_1} \\ 
%&=\int_{u_1}^{v_1}\dd x_1\;\Big(\int_{\R^{p-1}}\dd\tilde\De_{\uu,\vv}\;F_{x_1}\Big)^{(\1_{\{1\}})} \\ 
%&=\int_{u_1}^{v_1}\dd x_1\;\int_{\R^{p-1}}\dd\tilde\De_{\uu,\vv}\;\big(F_{x_1}\big)^{(\1_{\{1\}})} \\ 
%&=\int_{u_1}^{v_1}\dd x_1\;\int_{\tilde\uu}^{\tilde\vv}\dd\tilde\xx\;f_{x_1}(\tilde\xx) 
%=	\int_\uu^\vv\dd\xx f(\xx). 
%\end{align*}
\end{proof}

%\newpage
\begin{proof}[Proof of Theorem~\ref{prop:series}]
%IP As in Theorem~\ref{th:}, suppose that $\nn\ge(m-1)\1$. 
Let 
\begin{equation}\label{eq:R_{m,f}}
	R_{m,f}(\nn):=R_m, 
\end{equation}
with $R_m$ as defined in \eqref{eq:R_m}.
%\begin{equation}\label{eq:R_{m,f}}
%	\text{$R_{m,f}(n)$ be defined as $R_m$ \eqref{eq:R_m} but with $n-1$ replaced there by $n$. } 
%\end{equation}
%%with $R_m$ as defined in \eqref{eq:R_m}.
Then, by \eqref{eq:R unif}, 
\begin{equation}\label{eq:R to R}
	R_{m,f}(\nn)\underset{\wedge\nn\to\infty}\longrightarrow R_{m,f}(\infty).  
\end{equation}
Let 
\begin{equation}\label{eq:A_m,F(n)}
	A_{m,F}(\nn):=\sum_{J\subseteq[p]}(-1)^{p-|J|}A^J_{m,F}(\nn)
	=\tA_{m,F}(\nn)+(-1)^p A^\emptyset_{m,F}(\nn), 
\end{equation}
in view of \eqref{eq:tA_m,F(n)}. 
By \eqref{eq:}--\eqref{eq:A_m}, Lemma~\ref{lem:FTC}, \eqref{eq:A_m,F(n)}, and \eqref{eq:G}, 
%and the linearity of $A_{m,F}(\nn)$ in $F$, 
\begin{equation}\label{eq:..}
\begin{aligned}
	\sum_{\kk=\0}^{\nn-\1}f(\kk)-A_{m_0,F}(\nn)+R_{m,f}(\nn)=&A_{m,F}(\nn)-A_{m_0,F}(\nn) \\ 
=& \sum_{J\subseteq[p]}(-1)^{p-|J|}\big(A^J_{m,F}(\nn)-A^J_{m_0,F}(\nn)\big) \\ 	
=& \sum_{J\subseteq[p]}(-1)^{p-|J|}
\big(A^J_{m,T_J}(\nn)-A^J_{m_0,T_J}(\nn)+A^J_{m,F-T_J}(\nn)-A^J_{m_0,F-T_J}(\nn)\big), 
\end{aligned}	
\end{equation}
where 
$T_J=T_{J,\nn,m_0,F}$ is the Taylor polynomial of order $2m_0-1$ for the function $F$ at the point $\nn\1_J-\1$, so that 
\begin{equation*}
	T_J(\xx)=\sum_{\|\aal\|\le2m_0-1}\frac{F^{(\aal)}(\nn\1_J-\1)}{\aal!}\,(\xx-\nn\1_J+\1)^\aal
\end{equation*}
for $\xx\in\R^p$. 

Consider the mononomial $P(\xx)=\xx^\aal$ of degree $\|\aal\|\le2m_0-1$, so that $P(\xx)=\prod_{r=1}^pP_r(x)$, where $P_r(x):=x^{\al_r}$. 

Take any $r=1,\dots,p$ and any $J\subseteq[p]$, and let $n_{r,J}:=n_r\ii\{r\in J\}$. Following the lines of the proof of Proposition 5.1 in \cite{euler-maclaurin-alt} for the case when $f=P_r'$ and $F=P_r$, so that the polynomial $T$ therein coincides with $F=P_r$, we see from \cite[(5.5) and (9.19)]{euler-maclaurin-alt} that 
\begin{equation*}
	\sum_{\be=1-m}^{m-1}\tau_{m,1+|\be|}\,P_r(n-1/2-\be/2)=G_{m,P_r}(n)=G_{m_0,P_r}(n)
	=\sum_{\be=1-m_0}^{m_0-1}\tau_{m_0,1+|\be|}\,P_r(n-1/2-\be/2)
\end{equation*}
for any $n\in\Z_+$. 
So, by \eqref{eq:G_m,alt1} and \eqref{eq:tau_j}, 
\begin{align*}
A^J_{m,P}(\nn)=&\sum_{\bbe=(1-m)\1}^{(m-1)\1}\tau_{m,\1+|\bbe|}\,P(\nn\1_J-\1/2-\bbe/2) \\ 
=&\sum_{\bbe=(1-m)\1}^{(m-1)\1}\prod_{r=1}^p\big(\tau_{m,\1+|\be_r|}\,P_r(n_{r,J}-1/2-\be_r/2)\big) \\ 
=&\prod_{r=1}^p\sum_{\be=1-m}^{m-1}\big(\tau_{m,\1+|\be|}\,P_r(n_{r,J}-1/2-\be/2)\big) \\ 
=&\prod_{r=1}^p\sum_{\be=1-m_0}^{m_0-1}\big(\tau_{m_0,\1+|\be|}\,P_r(n_{r,J}-1/2-\be/2)\big)
=A^J_{m_0,P}(\nn).  
\end{align*}
Since $T_J$ is a polynomial of degree $\le 2m_0-1$ and $A^J_{m,F}(\nn)$ is linear in $F$, we conclude that 
\begin{equation}\label{eq:A_T-A_T}
	A^J_{m,T_J}(\nn)-A^J_{m_0,T_J}(\nn)=0\quad\text{for all $J\subseteq[p]$}.
\end{equation}

Further, the remainder $(F-T_J)(\nn\1_J-\1+\uu)$ at point $\nn\1_J-\1+\uu$ of the Taylor approximation $T_J$ of $F$ at $\nn\1_J-\1$ equals (cf.\ \eqref{eq:taylor}) 
\begin{equation*}%\label{eq:taylor}
\sum_{\|\aal\|=2m_0}\frac{2m_0}{\aal!}\,\uu^\aal\,\int_0^1\dd s\,(1-s)^{2m_0-1}F^{(\aal)}(\nn\1_J-\1+s\uu), 
\end{equation*}
which, by \eqref{eq:f^ to0}, goes to $0$ as $\wedge\nn\to\infty$ unless $J=\emptyset$. 
So, by \eqref{eq:G}, 
\begin{equation*}
	A^J_{m,F-T_J}(\nn)\underset{\wedge\nn\to\infty}\longrightarrow0\quad\text{and}\quad
	A^J_{m_0,F-T_J}(\nn)\underset{\wedge\nn\to\infty}\longrightarrow0\quad\text{unless $J=\emptyset$}. 
\end{equation*}
It follows now by %\eqref{eq:..}, 
\eqref{eq:A_T-A_T}, \eqref{eq:A^empty}, the linearity of $A^J_{m,F}$ in $F$, and (again) \eqref{eq:A_T-A_T} that the limit of the last expression in \eqref{eq:..} as $\wedge\nn\to\infty$ equals
\begin{align*}
	(-1)^p\big(A^\emptyset_{m,F-T_\emptyset}(\0)-A^\emptyset_{m_0,F-T_\emptyset}(\0)\big)
	&=(-1)^p\big(A^\emptyset_{m,F}(\0)-A^\emptyset_{m_0,F}(\0)\big)
	-(-1)^p\big(A^\emptyset_{m,T_\emptyset}(\0)-A^\emptyset_{m_0,T_\emptyset}(\0)\big) \\  
	&=(-1)^p\big(A^\emptyset_{m,F}(\0)-A^\emptyset_{m_0,F}(\0)\big). 
\end{align*}
Now \eqref{eq:series} follows, in view of \eqref{eq:R to R} and (the second equality in) \eqref{eq:A_m,F(n)}. 

Inequality \eqref{eq:R<<,infty} follows immediately from \eqref{eq:<M} and \eqref{eq:R<}--\eqref{eq:R<<}. 

Formula \eqref{eq:A^empty} follows immediately from \eqref{eq:G_m,alt2}. 
 
Theorem~\ref{prop:series} is completely proved. 
\end{proof}

\begin{proof}[Proof of Theorem~\ref{th:c}]
Note that 
\begin{equation}\label{eq:incl-excl}
\begin{aligned}
	\sum_{\kk=\0}^{\cc-\1}f(\kk)&=\sum_{\kk\ge\0}f(\kk)\ii\{\kk\le\cc-\1\} \\ 
	&=\sum_{\kk\ge\0}f(\kk)\prod_{r=1}^p\big(\ii\{k_r\le n_r+c_r-1\}-\ii\{c_r\le k_r\le n_r+c_r-1\}\big) \\  
&	=\sum_{\kk\ge\0}f(\kk)\sum_{J\subseteq[p]}(-1)^{|J|}
	\ii\{k_r\le n_r+c_r-1\ \forall r\in[p]\setminus J,\ c_r\le k_r\le n_r+c_r-1\ \forall r\in J\} \\ 
&	=\sum_{J\subseteq[p]}(-1)^{|J|}\sum_{\kk=\cc\1_J}^{\nn+\cc-\1}f(\kk)  
=\sum_{\kk=\0}^{\nn+\cc-\1}f(\kk)
+\sum_{\emptyset\ne J\subseteq[p]}(-1)^{|J|}\sum_{\kk=\cc\1_J}^{\nn+\cc-\1}f(\kk). 
\end{aligned}	
\end{equation}
Hence, 
\begin{equation}\label{eq:c}
\begin{aligned}
\sum_{\kk=\0}^{\nn+\cc-\1}f(\kk) - \tA_{m_0,F}(\nn+\cc)
&=\sum_{\kk=\0}^{\cc-\1}f(\kk)
-\sum_{\emptyset\ne J\subseteq[p]}(-1)^{|J|}\sum_{\kk=\cc\1_J}^{\nn+\cc-\1}f(\kk)-\tA_{m_0,F}(\nn+\cc) \\ 
&=\sum_{\kk=\0}^{\cc-\1}f(\kk)
-\sum_{\emptyset\ne J\subseteq[p]}(-1)^{|J|}
\Big(\sum_{\kk=\0}^{\nn+\cc-\cc\1_J-\1}f_{\cc\1_J}(\kk)-\tA_{m_0,F_{\cc\1_J}}(\nn+\cc-\cc\1_J) \Big)
+\RR,  
\end{aligned}	
\end{equation}
where 
\begin{equation*}
	\RR:=-\sum_{\emptyset\ne J\subseteq[p]}(-1)^{|J|}\tA_{m_0,F_{\cc\1_J}}(\nn+\cc-\cc\1_J)
	-\tA_{m_0,F}(\nn+\cc). 
\end{equation*}
By Lemma~\ref{lem:FTC} with $F=1$ (and $f=0$), 
\begin{equation}\label{eq:=0,-1}
	\sum_{J\subseteq[p]}(-1)^{|J|}=0\quad\text{and hence}
	\sum_{\emptyset\ne J\subseteq[p]}(-1)^{|J|}=-1. 
\end{equation}
Therefore and in view of \eqref{eq:tA_m,F(n)} and \eqref{eq:G_m,alt1},  
\begin{equation*}
	\RR=\sum_{\emptyset\ne J\subseteq[p]}(-1)^{|J|}\RR_J, 
%	=\sum_{\bbe=(1-m)\1}^{(m-1)\1}\tau_{m,\1+|\bbe|}\,
%	\sum_{\emptyset\ne K\subseteq[p]}(-1)^{p-|K|}\RR_{\bbe,K},
\end{equation*}
where 
\begin{equation*}
	\RR_J:=\tA_{m_0,F}(\nn+\cc)-\tA_{m_0,F_{\cc\1_J}}(\nn+\cc-\cc\1_J)
	=\sum_{\bbe=(1-m)\1}^{(m-1)\1}\tau_{m,\1+|\bbe|}\,\RR_{J,\bbe}, 
\end{equation*}
\begin{equation*}
	\RR_{J,\bbe}:=\sum_{\emptyset\ne K\subseteq[p]}(-1)^{p-|K|}
	\big[H\big((\nn+\cc)\1_K\big)-H\big(\cc\1_J+(\nn+\cc-\cc\1_J)\1_K\big)\big], 
\end{equation*}
and $H(\xx):=F(\xx-\1/2-\bbe/2)$. 
Thus, 
\begin{equation}\label{eq:RR}
		\RR
	=\sum_{\bbe=(1-m)\1}^{(m-1)\1}\tau_{m,\1+|\bbe|}\,
	\sum_{\emptyset\ne K\subseteq[p]}(-1)^{p-|K|}\RR_{\bbe,K},
\end{equation}
where 
\begin{equation}\label{eq:R_be,K}
\begin{aligned}
	\RR_{\bbe,K}
	:=&\sum_{\emptyset\ne J\subseteq[p]}(-1)^{|J|}
	\big[H\big((\nn+\cc)\1_K\big)-H\big(\cc\1_J+(\nn+\cc-\cc\1_J)\1_K\big)\big] \\ 
	=&\sum_{\emptyset\ne J\subseteq[p]}(-1)^{|J|}
	\big[H\big((\nn+\cc)\1_K\big)-H\big((\nn+\cc)\1_K+\cc\1_{J\setminus K}\big)\big] \\  
	=&\sum_{L\in\LL_K}
	\big[H\big((\nn+\cc)\1_K\big)-H\big((\nn+\cc)\1_K+\cc\1_L\big)\big]
	\sum_{J\in\JJ_{K,L}}(-1)^{|J|},   
\end{aligned}	
\end{equation}
\begin{equation*}
	\LL_K:=\{L\colon L\subseteq[p],\ L\ne\emptyset,\ L\cap K=\emptyset\},\quad
	\JJ_{K,L}:=\{J\colon \emptyset\ne J\subseteq[p],\ J\setminus K=L\}. 
\end{equation*}
For any $K\subseteq[p]$ and any $L\in\LL_K$, 
the map $J\mapsto I_J:=J\cap K$ is a bijection of the set $\JJ_{K,L}$ onto the set $\{I\colon I\subseteq K\}$, and for any $J\in\JJ_{K,L}$ the set $J$ is the disjoint union of the sets $I_J$ and $L$, so that $|J|=|I_J|+|L|$. 
It follows by \eqref{eq:=0,-1} that for any $K\subseteq[p]$ and any $L\in\LL_K$ one has 
$\sum_{J\in\JJ_{K,L}}(-1)^{|J|}=\sum_{I\subseteq K}(-1)^{|I|}(-1)^{|L|}=0$. 
Looking back at \eqref{eq:R_be,K} and \eqref{eq:RR}, we see that $\RR=0$. 
 
Letting now $\wedge\nn\to\infty$ and recalling \eqref{eq:c}, 
%the definition of $\sum_{\kk\ge\0}^\Alt f(\kk)$ in 
\eqref{eq:series}, the definition \eqref{eq:R_c} of $R_{m,f,\cc}(\infty)$, and formulas \eqref{eq:A^empty}, \eqref{eq:G_m,alt2}, and \eqref{eq:tA_m,F(n)},  
we have 
\begin{equation*}%\label{eq:c}
\begin{aligned}
\sum_{\kk\ge\0}^\Alt f(\kk) - \sum_{\kk=\0}^{\cc-\1}f(\kk)
&=-\sum_{\emptyset\ne J\subseteq[p]}(-1)^{|J|}\sum_{\kk\ge\0}^\Alt f_{\cc\1_J}(\kk) \\ 
&=-\sum_{\emptyset\ne J\subseteq[p]}(-1)^{|J|}
\big[(-1)^pA^\emptyset_{m,F_{\cc\1_J}}(\0)-R_{m,f_{\cc\1_J}}(\infty)\big] \\   
&=-R_{m,f,\cc}(\infty)
-\sum_{\emptyset\ne J\subseteq[p]}(-1)^{p-|J|}
\sum_{\aal=\0}^{(m-1)\1}\tau_{m,\1+\aal}\,\sum_{\bbe\colon|\bbe|=\aal}\,
F(\cc\1_J+\bbe/2-\1/2) \\   
&=-R_{m,f,\cc}(\infty)
-\tA_{m,F}(\cc),    
\end{aligned}	
\end{equation*}
which completes the proof of Theorem~\ref{th:c}. 
\end{proof}

\begin{proof}[Proof of Proposition~\ref{prop:C-decomp}]
Let $\aa_1,\dots,\aa_p$ denote the columns of the matrix $A$, so that $\aa_i\in\Z^p$ for each $i\in[p]$ and 
\begin{equation*}
	C:=A\R^+_J=\sum_{i\in[p]}\R^+_{\vp_i}\aa_i,\quad\text{where}\quad
	\vp_i:=%IP10-20-17
	1-\lb J\rb(i). 
\end{equation*}
%Without loss of generality (w.l.o.g.), for each $i\in[p]$ the $\gcd$ of the coordinates of the vector $\aa_i$ in the standard basis of $\R^p$ is $1$. 

If the matrix $A$ is unimodular, there is nothing to prove. So, %IP without loss of generality (w.l.o.g.)
w.l.o.g., $|\det A|\ge2$. Then there is a %\emph{nonzero} 
vector $\ww\in\Z^p\setminus\{\0\}$ such that 
\begin{equation}\label{eq:w}
	\ww=w_1\aa_1+\dots+w_p\aa_p
\end{equation}
for some real numbers $w_1,\dots,w_p$ in the interval $[0,1)$ (in fact, there are exactly $|\det A|-1$ such vectors $\ww$). 
%Since $\ww\ne\0$, at least one of the coordinates $w_1,\dots,w_p$ of the vector $\ww$ in the basis $(\aa_1+\dots,\aa_p)$ is nonzero. Actually, at least two of these coordinates must be nonzero, as follows from 
%
%\begin{lemma}\label{lem:ge2}
%Let $\aa=(a_1,\dots,a_p)$ and $\vv=(v_1,\dots,v_p)$ be any two vectors in $\Z^p$ such that $\vv=\la\aa$ for some real $\la>0$. Suppose also that $\gcd(a_1,\dots,a_p)=1$. 
%Then $\la\ge1$. 
%\end{lemma}
%
%\begin{proof}
%Clearly, $\la$ is rational, so that $\la=v/a$ for some natural numbers $a$ and $v$ with $\gcd(a,v)=1$, and $a_iv=b_ia$ for all $i\in[p]$. Hence, $a$ is a common divisor of all $a_i$'s, and so, $a=1\le b$, which implies $\la\ge1$.   
%\end{proof}
% 
Thus, w.l.o.g.\ %$p\ge2$ and 
for some %$k\in\{2,\dots,p\}$ 
$k\in[p]$ one has 
\begin{equation}\label{eq:k}
	\text{$w_j>0$ for $j\in[k]$\quad and\quad $w_j=0$ for $j\in[p]\setminus[k]$. }
\end{equation}
For each $i\in[k]$, let $A_i$ be the (integral) matrix obtained from the matrix $A$ by replacing its $i$th column, $\aa_i$, by $\ww$; then $\det A_i=w_i\det A$ and hence 
\begin{equation}\label{eq:A_i}
	0<|\det A_i|<|\det A|. 
\end{equation}
We shall see that \eqref{eq:C-decomp} holds with $I=[k]$, the matrices $A_i$ just defined, and some subsets $J_1,\dots,J_k$ of the set $[p]$. 

Then, repeating %if necessary 
the step described in the last paragraph -- for each of the matrices $A_1,\dots,A_k$ in place of $A$, in view of \eqref{eq:A_i} we shall eventually obtain \eqref{eq:C-decomp} with unimodular $p\times p$ matrices $A_i$ over $\Z$, as required. 
This step relies mainly on the following combinatorial lemma. 

\begin{lemma}\label{lem:C-decomp}
Let $\aa_1,\dots,\aa_p$, $C$, $\ww$, and $k$ be as described above. For each $i\in[k]$, let 
\begin{equation}\label{eq:C_i:=}
	C_i:=\R^+_{\vp_{ii}}\ww+\sum_{j\in[p]\setminus\{i\}}\R^+_{\vp_{ij}}\aa_j,    
\end{equation}
where the $\vp_{ij}$'s are any numbers in the set $\{0,1\}$ satisfying the following conditions: 
\begin{enumerate}[(i)]
	\item \label{rest} $\vp_{ij}=\vp_j$ for $i\in[k]$ and $j\in[p]\setminus[k]$; 
	\item \label{diag} $\vp_{ii}=\vp_i$ for $i\in[k]$; 
	\item \label{symm} $\vp_{ij}+\vp_{ji}=1$ for any distinct $i$ and $j$ in $[k]$; 
	\item \label{less} for each $i\in[k]$, the condition $\vp_i=1$ implies $\vp_{ij}\le\vp_j$ 
	for all $j\in[k]$;
	\item \label{=1} for each nonempty subset $J$ of the set $[k]$, there is some $i\in J$ such that for all $j\in J\setminus\{i\}$ one has $\vp_{ij}=1$. 
\end{enumerate}
Then 
\begin{equation}\label{eq:C_i-decomp}
	\lb C\rb=\sum_{i\in[k]}\lb C_i\rb.
\end{equation}
\end{lemma}

We also have 
\begin{lemma}\label{lem:eps}
Take any $\vp_1,\dots,\vp_p$ in $\{0,1\}$ and any $k\in[p]$. Then there exist numbers $\vp_{ij}$ in the set $\{0,1\}$ satisfying all the conditions \eqref{rest}--\eqref{=1} in Lemma~\ref{lem:C-decomp}. 
\end{lemma}

We shall prove these two lemmas in a moment. 

Letting now $J_i=\{j\in[p]\colon\vp_{ij}=%IP10-29-17 1
0\}$ for each $i\in[k]$ (so that $\vp_{ij}=%IP10-29-17
1-\lb J_i\rb(j)$ for all $i\in[k]$ and $j\in[p]$), we will 
have $C_i=A_i\R^+_{J_i}$ for $i\in[k]$, which will 
complete the step described %in the sentence right after formula \eqref{eq:A_i}
in the paragraph containing formulas \eqref{eq:w}--\eqref{eq:A_i}. Thus, to complete the entire proof of Proposition~\ref{prop:C-decomp}, it remains to prove Lemmas~\ref{lem:C-decomp} and \ref{lem:eps}.

%\newpage 

\begin{proof}[Proof of Lemma~\ref{lem:C-decomp}]
Take any $\xx\in\R^p$. Let $(y_1,\dots,y_p)=(y_1(\xx),\dots,y_p(\xx))$ denote the $p$-tuple of the coordinates of the vector $\xx$ in the basis $(\aa_1,\dots,\aa_p)$ of $\R^p$, so that 
\begin{equation}\label{eq:xx=}
	\xx=\sum_{j\in[p]}y_j\aa_j. 
\end{equation}
Also, for each $i\in[k]$, let $(y_{i1},\dots,y_{ip})=(y_{i1}(\xx),\dots,y_{ip}(\xx))$ denote the $p$-tuple of the coordinates of the vector $\xx$ in the basis $(\aa_1,\dots,\aa_{i-1},\ww,\aa_{i+1},\dots,\aa_p)$ of $\R^p$, so that 
\begin{equation*}
	\xx=y_{ii}\ww+\sum_{j\in[p]\setminus\{i\}}y_{ij}\aa_j
	=y_{ii}w_i\aa_i+\sum_{j\in[p]\setminus\{i\}}(y_{ii}w_j+y_{ij})\aa_j. 
\end{equation*}
In view of \eqref{eq:w} and \eqref{eq:k}, 
\begin{equation}\label{eq:vp_i=vp_ij}
	y_{ij}=y_j\quad\text{for } i\in[k],\ j\in[p]\setminus[k]. 
\end{equation}
As for $i$ and $j$ in $[k]$, we have $y_i=y_{ii}w_i$ and $y_j=y_{ii}w_j+y_{ij}
=\frac{y_i}{w_i}\,w_j+y_{ij}$ if $j\ne i$, which can be rewritten as 
\begin{equation}\label{eq:r}
\forall(i,j)\in[k]\times[k]\ 
	\Big(y_{ii}w_i=y_i\quad\text{and}\quad
	j\ne i\implies \frac{y_{ij}}{w_j}=r_j-r_i\Big), 
\end{equation}
where 
\begin{equation*}
	r_j:=r_j(\xx):=\frac{y_j}{w_j}. 
\end{equation*}

Note that \eqref{eq:C_i-decomp} means precisely that $C$ is the disjoint union of the $C_i$'s. Thus, the proof of Lemma~\ref{lem:C-decomp} will be completed in the following three steps. 

\textbf{Step 1: checking that $C_i\subseteq C$ for each $i\in[k]$.} Take indeed any $i\in[k]$, and then take any $\xx\in C_i$, so that, by \eqref{eq:C_i:=}, $y_{ij}\in\R^+_{\vp_{ij}}$ for all $j\in[p]$.  
Then $y_{ii}\in\R^+_{\vp_{ii}}=\R^+_{\vp_i}$ by condition \eqref{diag} of Lemma~\ref{lem:C-decomp} and hence $y_i=y_{ii}w_i\in\R^+_{\vp_i}$. 
Also, by \eqref{eq:vp_i=vp_ij} and condition \eqref{rest} of Lemma~\ref{lem:C-decomp}, $y_j=y_{ij}\in\R^+_{\vp_{ij}}=\R^+_{\vp_j}$ for $j\in[p]\setminus[k]$. 

If $y_i>0$, then $y_j=\frac{y_i}{w_i}\,w_j+y_{ij}>y_{ij}\ge0$ for all $j\in[k]\setminus\{i\}$, whence $y_j>0$ for all $j\in[k]$. 
So, by \eqref{eq:xx=}, 
% so that, in view of \eqref{eq:vp_i=vp_ij},  
$\xx\in\sum_{j\in[k]}\R^+_0\aa_j+\sum_{j\in[p]\setminus[k]}\R^+_{\vp_j}\aa_j
\subseteq\sum_{j\in[p]}\R^+_{\vp_j}\aa_j=C$. 

If now $y_i=0$, then the mentioned condition $y_i\in\R^+_{\vp_i}$ implies $\vp_i=1$. So, by condition \eqref{less} of Lemma~\ref{lem:C-decomp}, for all $j\in[k]$ we have $\vp_{ij}\le\vp_j$ and hence $\R^+_{\vp_{ij}}\subseteq\R^+_{\vp_j}$, which yields  
$y_j=\frac{y_i}{w_i}\,w_j+y_{ij}=y_{ij}\in\R^+_{\vp_{ij}}\subseteq\R^+_{\vp_j}$. 
So, in this case as well, $\xx\in C$. 

\textbf{Step 2: checking that the $C_i$'s are disjoint.} Take any distinct $i$ and $j$ in $[k]$, and then take any $\xx\in C_{i}\cap C_{j}$. 
Then $y_{ij}\in\R^+_{\vp_{ij}}$,  
whence, by \eqref{eq:r}, $r_{j}-r_{i}=y_{ij}/w_{j}\in\R^+_{\vp_{ij}}$. Similarly,  
$r_{i}-r_{j}\in\R^+_{\vp_{ji}}$, that is, 
$r_{j}-r_{i}\in-\R^+_{\vp_{ji}}=\R\setminus\R^+_{\vp_{ij}}$, by condition \eqref{symm} of Lemma~\ref{lem:C-decomp}. Thus, 
$r_{j}-r_{i}\in\R^+_{\vp_{ij}}\cap%\break 
\big(\R\setminus\R^+_{\vp_{ij}}\big)=\emptyset$, 
which is a contradiction. 

\textbf{Step 3: checking that $C\subseteq\bigcup_{i\in[k]}C_i$.} Take any $\xx\in C$, so that $y_j\in\R^+_{\vp_j}$ for all $j\in[p]$. 
Let 
\begin{equation*}
	J_\xx:=\{i\in[k]\colon r_i(\xx)\le r_j(\xx)\ \forall j\in[k]\}. 
\end{equation*}
Then, by \eqref{eq:r}, $y_{ij}\ge0$ for all $i\in J_\xx$ and $j\in[k]$. 
Moreover,  
$r_j(\xx)>r_i(\xx)%=y_i/w_i\ge0
$ for all $i\in J_\xx$ and $j\in[k]\setminus J_\xx$.
%, and 
%$r_j(\xx)=r_i(\xx)$ for all $i\in J_\xx$ and $j\in J_\xx$. 
%
So, again by \eqref{eq:r}, for all $i\in J_\xx$ and $j\in[k]\setminus J_\xx$ we have $y_{ij}>0$, so that $y_{ij}\in\R^+_0\subseteq\R^+_{\vp_{ij}}$. 
%, and $y_{ij}=0$ for all $i\in J_\xx$ and $j\in J_\xx\setminus\{i\}$. 
Note that $J_\xx\ne\emptyset$. 
So, by condition \eqref{=1} of Lemma~\ref{lem:C-decomp},  
there is some $i_\xx\in J_\xx$ such that for all $j\in J_\xx\setminus\{i_\xx\}$ one has $\vp_{i_\xx j}=1$, so that $y_{i_\xx j}\in\R^+_{\vp_{i_\xx j}}$. 
Thus, $y_{i_\xx j}\in\R^+_{\vp_{i_\xx j}}$ for all $j\in[k]\setminus\{i_\xx\}$. 
Also, $y_{i_\xx i_\xx}\in\R^+_{\vp_{i_\xx i_\xx}}$ -- in view of the first equality in \eqref{eq:r}, the condition $y_i\in\R^+_{\vp_i}$ for all $i\in[p]$, and condition \eqref{diag} of Lemma~\ref{lem:C-decomp}. 
Moreover, $y_{i_\xx j}\in\R^+_{\vp_{i_\xx j}}$ for all $j\in[p]\setminus[k]$ -- in view of \eqref{eq:vp_i=vp_ij}, the condition $y_j\in\R^+_{\vp_j}$ for all $j\in[p]$, and condition \eqref{rest} of Lemma~\ref{lem:C-decomp}.    
We conclude that $y_{i_\xx j}\in\R^+_{\vp_{i_\xx j}}$ for all $j\in[p]$, that is, $\xx\in C_{i_\xx}\subseteq\bigcup_{i\in[k]}C_i$. 

Lemma~\ref{lem:C-decomp} is now proved. 
\end{proof}

\begin{proof}[Proof of Lemma~\ref{lem:eps}]
For $i\in[k]$ and $j\in[p]\setminus[k]$, let $\vp_{ij}:=\vp_j$, in accordance with condition \eqref{rest} of Lemma~\ref{lem:C-decomp}. 

Similarly, let $\vp_{ii}:=\vp_i$ for $i\in[k]$, in accordance with condition \eqref{diag} of Lemma~\ref{lem:C-decomp}.

Next, w.l.o.g.\ $\vp_j$ is nondecreasing in $j\in[k]$. Let then 
$\vp_{ij}:=1$ and $\vp_{ji}:=0$ for all $i$ and $j$ in $[k]$ with $i<j$. 

It is now straightforward to check that all the conditions \eqref{rest}--\eqref{=1} in Lemma~\ref{lem:C-decomp} hold. 
In particular, concerning condition \eqref{less}, note that, if $\vp_i=1$ and $\vp_{ij}=1$ for some distinct $i$ and $j$ in $[k]$, then $i<j$ and hence $1=\vp_i\le\vp_j$, so that $\vp_j=1$. 
Concerning condition \eqref{=1}, for each nonempty subset $J$ of the set $[k]$, let $i:=\min J$; then 
for all $j\in J\setminus\{i\}$ one has $i<j$ and hence 
$\vp_{ij}=1$. % $\forall j\in J\setminus\{i\}$.  
Lemma~\ref{lem:eps} is now proved.   
\end{proof}

The entire proof of Proposition~\ref{prop:C-decomp} is thus complete. 
\end{proof} 

%\begin{acknowledgements}
%If you'd like to thank anyone, place your comments here
%and remove the percent signs.
%\end{acknowledgements}

% BibTeX users please use one of
%\bibliographystyle{spbasic}      % basic style, author-year citations
\bibliographystyle{abbrv}      % mathematics and physical sciences
%\bibliographystyle{spphys}       % APS-like style for physics
%\bibliography{}   % name your BibTeX data base

%\bibliographystyle{abbrv}
%%%%%\bibliographystyle{ims}
%%%%%\bibliography{are.citations}
%%%%%\bibliography{citat}
%%%%
%%%%%\bibliography{citations}
%%%%
\bibliography{P:/pCloudSync/mtu_pCloud_02-02-17/bib_files/citations12.13.12}

\def\cprime{$'$} \def\polhk#1{\setbox0=\hbox{#1}{\ooalign{\hidewidth
  \lower1.5ex\hbox{`}\hidewidth\crcr\unhbox0}}}
  \def\polhk#1{\setbox0=\hbox{#1}{\ooalign{\hidewidth
  \lower1.5ex\hbox{`}\hidewidth\crcr\unhbox0}}}
  \def\polhk#1{\setbox0=\hbox{#1}{\ooalign{\hidewidth
  \lower1.5ex\hbox{`}\hidewidth\crcr\unhbox0}}} \def\cprime{$'$}
  \def\polhk#1{\setbox0=\hbox{#1}{\ooalign{\hidewidth
  \lower1.5ex\hbox{`}\hidewidth\crcr\unhbox0}}} \def\cprime{$'$}
  \def\polhk#1{\setbox0=\hbox{#1}{\ooalign{\hidewidth
  \lower1.5ex\hbox{`}\hidewidth\crcr\unhbox0}}} \def\cprime{$'$}
  \def\cprime{$'$}
\begin{thebibliography}{10}

\bibitem{agap-godinho07}
J.~Agapito and L.~Godinho.
\newblock New polytope decompositions and {E}uler-{M}aclaurin formulas for
  simple integral polytopes.
\newblock {\em Adv. Math.}, 214(1):379--416, 2007.

\bibitem{agap-godin16}
J.~Agapito and L.~Godinho.
\newblock Cone decompositions of non-simple polytopes.
\newblock {\em J. Symplectic Geom.}, 14(3):737--766, 2016.

\bibitem{barv94}
A.~I. Barvinok.
\newblock A polynomial time algorithm for counting integral points in polyhedra
  when the dimension is fixed.
\newblock {\em Math. Oper. Res.}, 19(4):769--779, 1994.

\bibitem{haase05}
C.~Haase.
\newblock Polar decomposition and {B}rion's theorem.
\newblock In {\em Integer points in polyhedra---geometry, number theory,
  algebra, optimization}, volume 374 of {\em Contemp. Math.}, pages 91--99.
  Amer. Math. Soc., Providence, RI, 2005.

\bibitem{hoermander}
L.~H\"ormander.
\newblock {\em Lectures on nonlinear hyperbolic differential equations},
  volume~26 of {\em Math\'ematiques \& Applications (Berlin) [Mathematics \&
  Applications]}.
\newblock Springer-Verlag, Berlin, 1997.

\bibitem{KSW_ProcNAS}
Y.~Karshon, S.~Sternberg, and J.~Weitsman.
\newblock The {E}uler-{M}aclaurin formula for simple integral polytopes.
\newblock {\em Proc. Natl. Acad. Sci. USA}, 100(2):426--433, 2003.

\bibitem{KSW_duke07}
Y.~Karshon, S.~Sternberg, and J.~Weitsman.
\newblock Euler-{M}aclaurin with remainder for a simple integral polytope.
\newblock {\em Duke Math. J.}, 130(3):401--434, 2005.

\bibitem{lawrence}
J.~Lawrence.
\newblock Polytope volume computation.
\newblock {\em Math. Comp.}, 57(195):259--271, 1991.

\bibitem{euler-maclaurin-alt}
I.~Pinelis.
\newblock An alternative to the {E}uler--{M}aclaurin formula: {A}pproximating
  sums by integrals only.
\newblock arXiv:1705.09159 [math.CA], 2015.

\bibitem{watson}
G.~N. Watson.
\newblock A note on {G}amma functions.
\newblock {\em Proc. Edinburgh Math. Soc. (2)}, 11(Edinburgh Math. Notes 42
  (misprinted 41) (1959)):7--9, 1958/1959.

\end{thebibliography}
%\bibliography{C:/Users/ipinelis/Documents/Sync/mtu_Sync_01-19-17/bib_files/citations12.13.12}
%\bibliography{C:/Users/iosif/Sync/mtu_Sync_01-19-17/bib_files/citations12.13.12}
%\bibliography{C:/Users/ipinelis/Dropbox/mtu/bib_files/citations12.13.12}
%\bibliography{C:/Users/Iosif/Dropbox/mtu/bib_files/citations12.13.12}

%% Non-BibTeX users please use
%\begin{thebibliography}{}
%%
%% and use \bibitem to create references. Consult the Instructions
%% for authors for reference list style.
%%
%\bibitem{RefJ}
%% Format for Journal Reference
%Author, Article title, Journal, Volume, page numbers (year)
%% Format for books
%\bibitem{RefB}
%Author, Book title, page numbers. Publisher, place (year)
%% etc
%\end{thebibliography}
\end{normalsize}

\end{document}